\documentclass[psamsfonts]{amsart}
\usepackage{amssymb,amsfonts}
\usepackage[all,arc]{xy}
\usepackage{enumerate}
\usepackage{mathrsfs}
\usepackage{pstricks}
\usepackage{pst-node}

\usepackage{times}
\usepackage[T1]{fontenc}
\usepackage{mathrsfs}
\usepackage{epsfig}
\usepackage{color}

\newtheorem{thm}{Theorem}[section]

\newtheorem{prop}[thm]{Proposition}
\newtheorem{lem}[thm]{Lemma}

\theoremstyle{definition}
\newtheorem{defn}[thm]{Definition}

\newtheorem{exmp}[thm]{Example}

\theoremstyle{remark}

\newcommand{\NN}{{\mathbb N}}
\newcommand{\ZZ}{{\mathbb Z}}
\newcommand{\QQ}{{\mathbb Q}}

\newcommand{\sA}{{\mathcal A}}
\newcommand{\fA}{{\mathbb A}}
\newcommand{\sB}{{\mathcal B}}
\newcommand{\sC}{{\mathcal C}}
\newcommand{\sF}{{\mathcal F}}
\newcommand{\sG}{{\mathcal G}}
\newcommand{\sH}{{\mathcal H}}
\newcommand{\sL}{{\mathcal L}}
\newcommand{\sE}{{\mathcal E}}

\newcommand{\sP}{{\mathcal P}}
\newcommand{\sS}{{\mathcal S}}
\newcommand{\sT}{{\mathcal T}}
\newcommand{\sZ}{{\mathcal Z}}

\newcommand{\bs}{{\bf s}}
\newcommand{\bw}{{\bf w}}
\newcommand{\Ker}{{\rm Ker}}

\numberwithin{equation}{section}

\newcommand{\stdnodesep}{3}
\newcommand{\doffset}{3pt}

\newcommand{\node}[3]{\rput{0}(#2){\ovalnode{#3#1}{\Large #1}}}

\newcommand{\aline}[3]{%
	\ncline[nodesepA=\stdnodesep,nodesepB=\stdnodesep]%
	{->}{#1}{#2}%
	\Aput{#3}%
}

\newcommand{\bline}[3]{%
	\ncline[nodesepA=\stdnodesep,nodesepB=\stdnodesep]%
	{->}{#1}{#2}%
	\Bput{#3}%
}

\newcommand{\dline}[4]{%
	\ncarc[nodesepA=\stdnodesep,nodesepB=\stdnodesep,offset=\doffset]%
	{->}{#1}{#2}%
	\Aput{#3}%
	\ncarc[nodesepA=\stdnodesep,nodesepB=\stdnodesep,offset=\doffset]%
	{->}{#2}{#1}%
	\Aput{#4}%
}

\newcommand{\bcircle}[3]{%
	\nccircle[angleA=#2,nodesepA=\stdnodesep]{->}{#1}{20pt}%
	\Bput{#3}%
}

\title{ Path sets  in one-sided  symbolic dynamics}

\author{William  Abram}
\address{Department of Mathematics, Hillsdale College,
Hillsdale, MI   49242,   USA}
\email{wabram@hillsdale.edu}

\author{Jeffrey C. Lagarias}
\address{Department of Mathematics, University of Michigan,
Ann Arbor, MI 48109-1043,USA}
\email{lagarias@umich.edu}

\date{December 10, 2013, v5.4}
\thanks{
Work of the  second author was supported by NSF grant DMS-1101373.
}
\begin{document}

\begin{abstract}
Path sets are  spaces of  one-sided infinite symbol sequences associated to pointed graphs
 $(\mathcal{G},v_0)$, which are edge-labeled directed graphs with a distinguished vertex $v_0$. 
  Such sets arise naturally as address labels in geometric fractal constructions and in other contexts.
 The resulting set of symbol sequences
 need not  be closed under the one-sided shift.
 This paper establishes basic properties  of the  structure and symbolic dynamics
of path sets, and shows that they 
are a strict generalization of  one-sided sofic shifts.
\end{abstract}

\maketitle

%\tableofcontents

%************************************************************************
%
%  section 1. Introduction
%
%
%************************************************************************
\section{Introduction}

This paper investigates   the concept of  a {\em path set},
which is a notion in one-sided symbolic dynamics. Path sets form
 an enlargement of
the  class of  one-sided  sofic shifts 
which  includes certain
additional  closed sets not invariant under the one-sided  shift.
Let $ \mathcal{A}^\mathbb{N}$ denote the full one-sided shift space on the finite alphabet $\mathcal{A}$,
 topologized with the product topology. It  is a compact, completely disconnected topological  space. 
 Path sets are  a distinguished class of  closed subsets of $ \mathcal{A}^\mathbb{N}$  constructed as follows.
We are given a finite directed graph $G$ with edges bearing labels from $\mathcal{A}$ with a marked
vertex $v$.
A  {\em path set } prescribed by the data $(G, v)$  is the collection of  one-sided infinite sequences of 
edge labels assigned to all infinite paths in a finite directed graph $G$
emanating  from the vertex $v$ of $G$. 
It is easy to see that path sets  are closed subsets of $\sA^{\NN}$.
We denote  the collection of all path sets on the  alphabet $\sA$
by  $\sC(\sA)$.

 One reason for studying path sets   is that they  naturally arise in 
   connection with geometric constructions of fractals and  limit sets
   of discrete groups. 
     Associated to these
   constructions are 
    {\em address 
   maps}, given by paths in finite directed graphs,
   which specify labels of points in the limit sets of such recursive
   constructions. 
    Such address maps 
       arise in 
geometric graph-directed constructions (Mauldin and Williams \cite{MW88},  Edgar \cite[Sect. 4.3]{Edg08}),
  in iterated function systems 
(Barnsley \cite[Sec. 4.1]{Ba93}),
 in describing limit sets of various discrete group actions (Mauldin and Urbanski \cite{MU03})
and in describing boundaries of fractal tiles (Akiyama and Lorident \cite{AL11}). 
 Under some conditions
   addresses are unique, but in other circumstances
   multiple addresses label the same  geometric point. There are known conditions,
   such as the ``open set condition" under which almost all point have a unique address
   (Bandt et al \cite{Ban06}).
   The symbolic dynamics objects we study here are  the complete
   sets of distinct addresses given by an address map, whether or not  addresses are unique.

   Path sets are not a new concept; they
have previously  appeared in 
the symbolic dynamics literature under the name  ``follower set,"
typically as an auxiliary construction.
The usual  framework of coding in symbolic dynamics, as given in Lind and Marcus \cite{LM95}, 
restricts to  shift-invariant sets (ones with $\sigma (X) \subset X$)  and emphasizes two-sided dynamics;
 shift-invariant sets in the one-sided case are considered in Kitchens \cite{Kit98} for shifts of finite type.
Path sets are objects in  one-sided dynamics that 
are not always  invariant under the one-sided shift map;  
 the initial condition imposed by the marked vertex typically breaks shift-invariance.
Their  distinctive properties 
 related to  lack of shift invariance seem not to have been studied in any detail.
 This paper proposes the terminology  {\em  path sets},  which is consistent with language
  used in  fractal constructions (\cite[Sec. 4.3]{Edg08}),   because   ``follower set"  is used 
 in the symbolic dynamics literature with several different
 meanings,  see the discussion in  Section  \ref{sec12}.  An alternate descriptive term for path set could be   
 {\em pointed follower set}.

 The object of this  paper  is to establish
 basic properties of path sets
 in a relatively self-contained manner. It particularly addresses those 
 properties connected to 
  symbolic dynamics and the action on them of the one-sided shift map
  $\sigma:  \mathcal{A}^\mathbb{N} \to  \mathcal{A}^\mathbb{N}$ which sends
 $\sigma( a_0 a_1a_ 2, \cdots) = (a_1a_2 a_3, \cdots)$, with each $a_i \in \sA$.
 %The treatment of path sets  necessitates small modifications of the basic symbolic
%dynamics framework  and changes in  hypotheses    
%and conclusions of some  basic results.
% relative to that framework.
% which this paper presents. 
We  show  that path sets  form a 
strict generalization of   one-sided sofic shifts on the alphabet $\sA$,
and  characterize sofic shifts as exactly the (one-sided) shift-invariant members of $\sC(\sA)$.
Members of  $\sC(\sA)$  retain  the  good set-theoretic and 
topological entropy properties of one-sided sofic shifts.
The enlargement from the class of sofic shifts to the class of path sets gives a class closed
under a larger set of operations than for sofic shifts; in particular
when the alphabet has $g$ symbols
$\sA= \{ 0,1, 2, ..., g-1\}$
the class $\sC(\sA)$ is closed under
the  $g$-adic arithmetic operations 
of adding and multiplying by ($g$-integral) rational numbers
in the sense of Mahler \cite{Mah61}, 
 see our paper \cite{AL11p}.

One-sided sofic shifts have two different types of graph presentations,
one as a sofic shift graph (represented by  paths starting from any vertex)
and a second as a path set graph (represented by paths starting from a fixed vertex).
There is typically a cost associated with encoding an initial condition in the presentation
of a sofic shift.  We show that  the minimal path set presentation  of
 a sofic system  (measured by number of vertices in the graph)  requires at
 most an  exponential increase 
 in its size relative to its minimal presentation as a sofic system, and present
 examples showing that such an exponential increase must sometimes occur. 
 We also  relate  the notion of path set  to
 several similar  concepts,  and discuss applications
 of this concept in  fractal constructions.

  %****************
% Section 1.1
%******************
\subsection{Main Results}\label{sec11}
A \emph{pointed graph} $(\mathcal{G},v)$
over a finite alphabet $\mathcal{A}$   comprises a
finite edge-labeled directed graph
  $\mathcal{G}= (G, \sE)$  and  a distinguished vertex $v$ of 
the underlying directed graph $G$. The directed graph  $G=(V, E) $ is specified by its vertex
and (directed)  edge sets $E$ with edges $e=(v_1, v_2) \in V \times V$, and the data 
$\sE\subset E \times \sA$  
specifies the set of labeled edges $(e, \ell)$,
with  labels  drawn from the  alphabet $\sA$. 
We allow loops and  multiple edges, but require that all triples $(e, \ell) = (v_1, v_2, \ell)$ be distinct.
The  results of this paper regard
the alphabet $\sA$ as fixed, unless specifically noted otherwise.

%************************
% Defi 1.1
%************************
\begin{defn} \label{de111}
The  \emph{path set} (or {\em pointed follower set}) 
   $\sP = X_\mathcal{G}(v) $ specified by a pointed graph $(\sG, v)$ 
  is the subset of $\sA^{\NN}$ made up  of  the symbol sequences
of  successive edge labels of all possible one-sided infinite walks in $\mathcal{G}$ 
issuing from the distinguished vertex $v$. 
We let $\sC(\sA)$ denote the collection of all path sets using labels from  the alphabet $\sA$.
Many different $(\mathcal{G}, v)$ may give the same path set $\sP \subset \sA^{\NN}$, and we call any such
$(\mathcal{G},v)$  a \emph{presentation} of $\sP$. 
\end{defn}

This definition detects  paths without counting multiplicity of occurrence.
Therefore without loss of generality  we may  suppose
 that the labelled graph $\sG$ has
the {\em non-redundancy property} that each
labeled edge  datum $( (v_1, v_2), \ell)$ 
occurs at most once in the graph. 

We first summarize behavior of path sets under set-theoretic operations.

%************************
% Theorem 1.2
%************************
\begin{thm} \label{topthm} {\em  (Set operations on Path Sets)}
 Path sets  on a fixed alphabet $\sA$ have the following properties.

(1) Each path set 
$\mathcal{P} \in \sC(\sA)$
 is a closed subset of $\mathcal{A}^\mathbb{N}$, carrying the product topology.

(2) If $\mathcal{P}_1$ and $\mathcal{P}_2$ are path sets, then so is $\mathcal{P}_1 \cap \mathcal{P}_2$.

(3) If $\mathcal{P}_1$ and $\mathcal{P}_2$ are path sets, then so is $\mathcal{P}_1 \cup \mathcal{P}_2$.
\end{thm}

 The complement $X \smallsetminus \sP$ of a path set $\sP$
 inside a given  $X = \sA^{\NN}$ with fixed alphabet $\sA$ need not be a path set.
Thus  the collection $\sC(\sA)$
of path sets with fixed alphabet $\sA$ 
does not form  a Boolean algebra of sets.

Second,  we study the action of the one-sided shift operator on path sets, and
in doing so relate path sets to sofic
shifts in symbolic dynamics. Recall that the one-sided shift map 
$\sigma \mathcal{A}^\mathbb{N} \to  \mathcal{A}^\mathbb{N} $ acts on the semi-infinite symbol sequences 
by
 $$\sigma( a_0a_1a_2a_3 \cdots) := (a_1a_2a_3 \cdots).$$
 The notion of a sofic system was originally  introduced by
Weiss \cite{Wei73} in 1973, in the context of two-sided sequences,  and much of the
literature treats the two-sided case, particularly  Lind and Marcus \cite[Chap. 3]{LM95}. 
Here we consider a   one-sided version
of this concept,  defined as follows.
%************************
% Defn. 1.3
%************************

\begin{defn}\label{de13}
A \emph{one-sided sofic shift} is a subset $Y \subset \mathcal{A}^\mathbb{N}$  
 specifying all possible sequences of  labels of one-sided infinite walks along a finite edge-labeled directed graph $\mathcal{G}$,
 starting from any vertex of the graph. 
 \end{defn}
 
 This definition immediately implies that a one-sided sofic shift $Y$ is invariant
 under the one-sided shift map, i.e.  $\sigma(Y) \subseteq Y$.
 
 Ashley, Kitchens and Stafford \cite{AKS92} previously 
introduced a  notion of one-sided sofic shift, 
defining it  as a symbolic dynamical system
having the property of  finiteness of
the collection of all  possible finite
follower sets (as defined in  Section \ref{sec12} below).
Appendix B of this paper shows that  their definition is
equivalent to  the definition above.

   We show that the  notion of a path set is a 
  strict generalization of  a one-sided sofic shift.
%************************
% Theorem 1.4
%************************

\begin{thm} \label{softhm} {\em (One-sided Shift action on Path Sets)}\\
\indent (1)  For any path set  $\mathcal{P}$, the shifted  set  $\sigma(\mathcal{P})$ is also a path set.

(2) The (one-sided) shift closure $\overline{\mathcal{P}} = \cup_{j \in \mathbb{N}} \sigma^j(\mathcal{P})$ of
a path set is  a path set.
 
  (3) Every shift-invariant path set is a one-sided sofic shift,  and conversely.
  
\end{thm}

This result raises a  realizability problem when treating
 sofic shifts as path sets.  The {\em realizability problem} is: Given a sofic shift $Y$ with a labeled graph presentation  $\sG$,  
 construct a pointed graph presentation $(\sH,v')$ for it as
a path set.  One such  presentation can be obtained by a standard  construction given in Theorem ~\ref{rrthm},
where we show that  every path set has a right-resolving presentation.
In this construction the  new graph presentation ($\sH, v')$ obtained  may be exponentially larger in size 
(as measured by the number of vertices)  than
the original presentation  $\mathcal{G}$ of the sofic shift.  
We show  that   exponential blow-up in size is sometimes unavoidable in general, 
 in that there exists a
 family of sofic shifts $Y$ whose  minimal right-resolving presentations
  as a path set $(\sH, v_0)$  requires  an exponentially larger number of states  than its
 minimal right-resolving presentation as a sofic shift, see Example \ref{ex34}.

Third, we show closure of path sets under a decimation (or fractionation) operation. 
Given $j \ge 0, m \ge 1$ and define
the {\em decimation map} $\psi_{j, m}: \sA^{\NN} \to \sA^{\NN}$ by
$$
\psi_{j, m}(a_0 a_1 a_ 2 \cdots) := (a_j a_{j+m} a_{j+2m} \cdots)
$$
The decimation  operation extracts the digits of the path set in a specified 
infinite arithmetic progression of indices.

%************************
% Theorem 1.5a
%************************

\begin{thm} \label{fracthm} {\em (Decimation of Path Sets)}\\
\indent  
(1) For any path set  $\mathcal{P}$, and any $(j, m)$ with  $j \ge 0$ and $ m \ge 1$,  the decimated  set
$$
\sP_{j, m} := \psi_{j, m}(\sP) = \{ \psi_{j, m}(x): ~ x \in \sP \}
$$ 
 is  a path set.
 
 (2) Suppose  that the path set $\mathcal{P}$ is shift-invariant.  Then every decimated set $\sP_{j, m}$
 is shift-invariant. In addition,  for fixed $m$ all the sets $\sP_{j, m}$ for $j \ge 0$ are equal.
 (We may use the abbreviated notation $\sP_m$ in this case.)
 \end{thm}
 
 We make some further definitions associated with decimation operations.
 %************************
% Defn. 1.6a
%************************
\begin{defn}\label{de16a}
For a fixed $m \ge 1$,  the {\em $m$-kernel} of 
 path set $\sP$ is  the collection of all path sets
$$
\Ker_m(\sP) := \{ \psi_{j, m^k}(\sP) :  j \ge 0, k \ge 0\}.
$$
\end{defn}

A priori, the  number of distinct path sets in this collection $\Ker_m (\sP)$ could be  finite or infinite,
depending on $\sP$.

%************************
% Defn. 1.7a
%************************
\begin{defn}\label{de17a}
A path set $\sP$ is an {\em $m$-automatic path set}
for a given $m \ge 2$ 
whenever its $m$-kernel $\Ker_m(\sP)$ is a finite collection of path sets. 
\end{defn}

This definition is formulated  in analogy to that  of {\em $m$-automatic sequence}, as  given in
Allouche and Shallit \cite{AS03}. The property of  being an $m$-automatic
sequence  is a property of a single infinite sequence (e.g. of a single path),
while our definition above concerns a property of a set of paths.
In our  terminology  an earlier result
of  Cobham \cite{Cob69} established:
A single infinite sequence is an $m$-automatic sequence (in the sense of \cite{AS03}) 
if and only if its $m$-kernel is finite
(cf. Allouche and Shallit \cite{AS92}). 
Note that in  general  the  shift closure of an $m$-automatic sequence
need not be a path set: The Thue-Morse sequence \cite[p. 152]{AS03} is a $2$-automatic sequence
with $\sA=\{0, 1\}$, but its shift closure, the {\em Morse shift} (cf. \cite[pp. 457--459]{LM95}), is not  a path set
\footnote{The Morse shift is known to be a minimal shift and to not
contain any periodic word. 
Each path set contains an eventually periodic path, so a 
shift-invariant path set always contains a periodic word.}.

The full shift on a fixed alphabet  is an $m$-automatic path set for every $m \ge 1.$
We leave it as an open problem to characterize all $m$-automatic path sets.
%%%%%%%%%
%Recent papers study various sets closed under decimation operations
%and determine their Hausdorff dimension (\cite{KPS11}, \cite{PSSS12}).
%%%%%%%%%%

Fourth, we characterize  path sets in terms of a finiteness property under the shift operation and intersection.

%************************
% Definiton 1.8
%************************

\begin{defn}
 For an alphabet $\sA$ and  $j \in \mathcal{A}$,
define the {\em prefix set}
   $\sZ_j := j \mathcal{A}^{\mathbb N}$
    to be the closed set of all sequences whose
  initial digit is $j$ and whose subsequent digits are arbitrary.
    \end{defn}

 It is easy to see that each prefix set  $\sZ_j$ is a path set.  Using these sets we obtain
the following structural characterization of path sets.

%************************
% Theorem 1.9
%************************

\begin{thm} \label{chartheorem} {\em (Structure Theorem for Path Sets)}
The following are equivalent.
\begin{enumerate}
\item
$\sP$ is a path set in $\sC(\sA)$.
\item
$\sP$ is a closed subset of $\mathcal{A}^\mathbb{N}$ that has  the property that
there is some finite collection 
of subsets of $\sP$ which contains $\sP$ and which is closed 
under the operations:
\begin{enumerate}
\item
apply the one-sided shift,
\item
 intersect  with a prefix set $\sZ_j$, for any $j \in \mathcal{A}$.
\end{enumerate}
\end{enumerate}
\end{thm}

This finiteness property differs from  the finiteness property imposed  in the
definition of  $m$-automatic path set above. 

Fifth,  we study the  notion of  entropy for path sets. 
Since these sets need not be  shift-invariant, a priori there  are  two  distinct 
notions of topological entropy, as follows.

%************************
% Definition 1.10
%************************

\begin{defn} \label{entdef} 
(1) The \emph{ path topological entropy} of a path set $\mathcal{P}$ is
\begin{equation} \label{Eqdef14} H_{p}(\mathcal{P}) := \limsup_{n \rightarrow \infty} \frac{1}{n} \log N^I_n(\mathcal{P}),
\end{equation}
where $N^I_n(\mathcal{P})$ denotes the number of distinct \emph{initial} blocks of length $n$ from $\mathcal{P}$.
\newline (2) The \emph{topological entropy} of a path set $\mathcal{P}$ is
\begin{equation} \label{htopeq} H_{top}(\mathcal{P}) := \limsup_{n \rightarrow \infty} \frac{1}{n} \log N_n(\mathcal{P}),
\end{equation}
where $N_n(\mathcal{P})$ is the number of distinct 
 blocks of length $n$ occurring anywhere in 
 the symbol sequences  in  $\mathcal{P}$.
\end{defn}

These definitions are  independent of the alphabet $\sA$ in which the path set $\sP$ is viewed
as belonging.  
The  definition (2) above 
corresponds to the definition of topological entropy for shift spaces (Adler and Marcus \cite{AM79}), given by the shift
closure $\overline{\mathcal{P}}.$
In fact  these two notions of topological entropy agree.
 
%************************
% Theorem 1.11
%************************

\begin{thm} \label{entthm} {\em (Topological Entropy Equivalence)}
For a path set $\mathcal{P}$,
\begin{equation} \label{enteq} H_p(\mathcal{P}) = H_{top}(\mathcal{P}) = H_{top}(\overline{\mathcal{P}}).
\end{equation}
All $\limsup$'s in the definitions are attained as limits.
\end{thm}

Since the topological entropy of a sofic shift is well understood, this allows us 
 to compute the path topological entropy $H_p(\mathcal{P})$, using a suitably
 chosen presentation of $\sP$. 
 A presentation $(\sG, v)$ of a path set  is {\em reachable} if each vertex of $\sG$ can be reached by
 a directed path from $v$.  We recall a standard definition.
 
 %************************
% Definiton 1.12
%************************

\begin{defn}
  A labeled directed graph $\sG$
is   {\em right-resolving} if from each vertex of $\sG$, all  the exit edges  have distinct labels.
  \end{defn}

 In Theorem \ref{rrthm} in Section \ref{sec3} we show that every path set has a 
 reachable   presentation that is right-resolving. Our final result gives
 conditions where  the
 standard formula for topological entropy for 
 shifts of finite type or sofic shifts extends to path sets.

 %************************
%  Theorem 1.13
%************************

\begin{thm} \label{Percor} {\em (Topological Entropy Formula  for Path Sets)}\\
Let $\mathcal{P}$ be a path set with reachable presentation 
$(\mathcal{G}, v)$ having a right-resolving  labeled  graph $\sG$. Then
\begin{equation} 
H_p(\mathcal{P}) = \log \lambda,
\end{equation}
where $\lambda$ is the spectral radius of the adjacency matrix $A$ of the underlying 
directed
graph $G$ of $\mathcal{G}$. If in addition $\mathcal{G}$ is irreducible, 
then so is $A$, and then $\lambda$ is the Perron eigenvalue of $A$.
\end{thm}

Here $A$ is the adjacency matrix of the unlabeled, directed graph $G$ underlying   $\mathcal{G}$.
It takes values in $\mathbb{N} \cup \{0\}$, so if it is irreducible it follows by the 
Perron-Frobenius Theorem that the Perron eigenvalue $\lambda$ is well defined.
 It is a positive real number whose magnitude is maximal among the eigenvalues of $A$.

%****************
% Section 1.2
%******************
\subsection{Related Work}\label{sec12}

There  is  a large literature on  concepts similar to path sets
 in automata theory, semigroups, and 
descriptive set theory.  The book of Perrin and Pin \cite{PP04} presents many results about
infinite paths in these contexts, viewing them as infinite words.  The concepts 
covered typically differ in their
level of generality from path sets. We mention some of these below, and also
indicate relations to terminology in these areas.
%%%%%%%%%%%%%%%%%%%%%%%%%%%%
%These concepts usually differ from path sets in their  level of generality.
%are given in  different terminology, and often have different
%set-theoretic closure properties. The following give some of them.
%%%%%%%%%%%%%%%%%%%%%%%%%%%%%

In automata theory   a finite directed labeled graph is called
a {\em finite automaton}. 
A {\em regular language} corresponds to  the set of  finite labeled   paths
produced as output by a finite automaton (directed labelled graph) with fixed initial state.
This theory goes back to Kleene and Moore; an interesting nonstandard treatment
is given in Conway \cite{Con71}.
Infinite paths in graphs have been considered in automata theory; treatments of this topic
appear in  Eilenberg \cite[Chap. 14]{Eil74A}, B\'{e}al and Perrin \cite{BP97},  and
Perrin and Pin \cite{PP04}. 
The treatment of Eilenberg \cite{Eil74A}  considers  a finite state automaton 
$\fA :=(Q, I, T)$ in which $Q$ is  a finite labeled directed graph, 
$I$ specifies a set of vertices called {\em initial states}, and $T$ specifies a set of 
vertices called {\em terminal states}. 
Various collections of infinite paths have been studied in this context.
In Chapter 14 Eilenberg introduces the set   $||\fA||$
of all {\em successful $\omega$-paths}, where a successful
$\omega$-path is any infinite path in $Q$
that starts at some  state in $I$ and visits some state in $T$ infinitely many times.
The notion  of path set $(\sG, v)$ defined in this paper corresponds to
the special case of this concept in which   $Q=\sG$,
the initial state $I = \{ v\},$ and the terminal state set $T$ is the set of  all states.
But this special case is not singled out in the automata theory literature, to our knowledge.

A more general notion in automata theory of {\em successful $\omega$-paths} accepts
only infinite paths that visit infinitely often exactly the vertices in one of a specified collection $\sT=\{ T_1, T_2, ..., T_m\}$
of subsets of vertices of $Q$.
The totality of  sets $||\sA||$ in
this generalized sense were characterized in 1966 by McNaughton \cite{McN66} 
(see  also \cite[Chap. 14]{Eil74A}, \cite[Sec. 3.2]{PP04});
they form a larger collection of sets which is closed under complement (unlike path sets).
In  Theorem 1.3 of Chapter 14 Eilenberg  classifies all $||\sA||$
in which both the initial  state $I$ and terminal state $T$ are allowed to vary.
In this case the set of all $||\sA||$ is shift-invariant (unlike path sets).

Turning to the  symbolic dynamics literature,  path
sets have appeared under the term ``follower set," for example
in Jonowska and Marcus \cite{JM94}.
However the  term  ``follower set'' has  been used  with at least two alternate meanings: 
  
  (1) One notion of  ``follower set" is 
  the set of all possible finite paths exiting a fixed vertex, see
   Lind and Marcus \cite[Definition 3.3.7]{LM95}.  We  may call this concept  a  {\em finite follower set}.
   (Note however that a finite follower set may contain infinitely many elements!)  
   A characterization of two-sided sofic shifts is given in terms of the finiteness of the
   collection of all possible finite (right) follower sets ( \cite[Theorem 3.2.10]{LM95}). In Appendix B we provide a
   similar characterization for one-sided sofic shifts in terms of finite follower sets.
  
  (2) Another notion  of ``follower set"  
     is   the set $\overline{F}(W)$  of all possible one-sided (infinite) paths $\alpha_0\alpha_1 \alpha_2...$  that can follow 
 a given finite word $ W = \alpha_{-k}\cdots  \alpha_{-3} \alpha_{-2}\alpha_{-1}$ in a two-sided   sofic shift $Y$.
 We  call this concept  an {\em infinite follower set}.
 Occurrences of this finite word may terminate in  several different vertices, so a priori this notion
 differs from the notion of path set in this paper. 
   Infinite follower sets  appear in  the
 work of Fischer \cite{Fis75}, \cite{Fis75b} 
 connecting   two-sided sofic shifts with  graphs.
 Further work was done  by Krieger \cite{Kr83}, \cite{Kr87}. 
 Using this concept Fischer \cite[Theorem 2]{Fis75b}
characterized  two-sided sofic shifts as those
shift-closed subsets of $\sA^{\mathbb{Z}}$ having only
a finite number of distinct  infinite follower sets.

   Theorem \ref{th61} of this paper shows
    that the  concepts of path set and infinite follower set
   are distinct:  for a fixed alphabet $\sA$ infinite follower sets 
 are always path sets, but  the converse assertion does not always hold.
    This result also  shows that  all path sets of an alphabet $\sA$ are representable as  infinite follower sets 
 on a larger  label  alphabet  $\sA^{'}$, which adds   one extra symbol.
 
  %****************
% Section 1.3
%******************
\subsection{Applications}\label{sec13}
Graph-directed constructions of fractals involve two separate ingredients: 
\begin{enumerate}
\item
 A family of geometric maps being iterated; 
 \item 
A graph (or automaton) describing the allowable combinations of maps under iteration.
\end{enumerate}
Both ingredients together determine the final geometric fractal object.   The
same  geometric fractal object typically has many quite different  constructions, and
these  constructions may attach different addresses to the points
in them. 
The  path set concept  captures   that part of the geometric construction
of fractal objects that occurs
purely at the level of the address map. In particular, it  facilitates study of situations
where the underlying maps used in a geometric graph-directed construction or IFS are
{\em  held fixed}, while the underlying graph giving the address map is {\em varied}.
For example  Theorem \ref{topthm}  shows that in this circumstance
intersections and unions of fractal sets 
can be computed purely at the level of address maps.

  Path sets  arise in 
 $p$-adic versions of fractal constructions. In the paper \cite{AL11p}
 we introduce a notion of {\em $p$-adic path set fractal}, which is the
 image of a path set interpreting its address as a $p$-adic integer.
 We show that  such sets are constructible by a $p$-adic graph-directed
 construction with a constant ratio similarity.  
 We also  show that the resulting collection of  $p$-adic path sets is closed under
 the $p$-adic arithmetic operations of adding $r$ to a set  or multiplying  the set by $r$,
 where   $r= \frac{m}{n} \in \QQ \cap \ZZ_p$, e.g. $r$ is a 
 {\em $p$-integral} rational number.   A need for the path set concept 
 is forced by these operations,  because even  if one restricts the arithmetic operations to 
 apply only to (one-sided) shift invariant sets, the image path sets under the operations
 need not be shift-invariant.
 An  interesting feature of the arithmetic constructions in \cite{AL11p}
 is that the  associated   graph presentations  of the new path sets 
 constructed using  arithmetic operations can be  far from
 irreducible, sometimes having nested irreducible subcomponents, cf. 
 examples in the related paper \cite{AL11b}.
 
 In special circumstances the  Hausdorff dimension
of   fractal objects may  be computed using
 the topological entropy formulas  given in this paper.
In graph-directed constructions the Hausdorff dimension of the
resulting geometric fractal  typically depends in a complicated
way  on both the geometric maps and
the address map.
In general these dimensions can be computed using operators that require
mixing together both these ingredients, e.g. Mauldin and Williams \cite{MW88}. 
In  the special case where all the maps are similarities with constant ratio,  
the Hausdorff dimension  is often determinable from the address
map data  together with only knowledge of  the similarity ratio,
under extra hypotheses such as variants of the  open set condition
(Ngai and Wang \cite{NW01}, Bandt \cite{Ban06}, Lau and Ngai \cite{LN07}).
These extra hypotheses hold in the $p$-adic path set fractal
framework used in  \cite{AL11p} and \cite{AL11b}, and 
we use Hausdorff dimension calculations  in \cite{AL11b} to answer a question
raised in \cite[Sec. 4]{Lag09},  motivated by 
a number-theoretic problem of Erd\H{o}s.

There are several possible directions for further work. 
First, the  decimation operation in Theorem \ref{fracthm} of this paper
gives  a construction of  new path sets from old, worthy of further study. 
Recent papers study various sets closed under decimation operations
and determine their Hausdorff dimension (\cite{KPS11}, \cite{PSSS12}).
Second,   Feng and Wang \cite{FW09} 
have recently considered  the question of
which iterated function systems (given using self-similar maps) will generate the same limiting 
geometric object $F$. It asks when a generating IFS family, all giving this object, has a minimal element.
This only happens in special situations.
An analogous question in our framework concerns the problem of when does a path set 
have a unique minimal presentations, measured by the number of vertices in the underlying graph.
Third, one may consider path sets in 
 the study of  general systems of numeration, viewed for example 
in  the fiber systems model for such systems, as described in 
  Barat, Berth\'{e}, Liardet
and Thuswaldner \cite{BBLT06}.

%************************
% Section 1.4
%*********************
\subsection{Contents}

Section \ref{sec2} presents  examples  illustrating some
 distinctive properties of path sets relative to one-sided sofic shifts.
Section \ref{sec3}  treats presentations of path sets and the question of how efficiently
they can be encoded relative to one-sided  sofic shifts.
 Section \ref{sec3b}  establishes the set-theoretic properties of path sets given in
  Theorem ~\ref{topthm}. 
Section \ref{sec4}  treats one-sided sofic shifts and  gives the proof of Theorem ~\ref{softhm}. 
Section \ref{sec4a} relates the notions of path sets and infinite follower sets for two-sided sofic shifts,
stated as Theorem \ref{th61}. It also makes remarks on minimality of presentations of sofic shifts and path sets.
Section \ref{sec5a} gives the proof of Theorem ~\ref{fracthm}.
Section \ref{sec6} proves the structural characterization in Theorem \ref{chartheorem}.
Section \ref{sec5} proves results on topological entropy, giving Theorem \ref{entthm} and Theorem \ref{Percor}. 
Appendix A (Section \ref{sec10}) gives a proof that Example \ref{ex33} in Section \ref{sec3} gives
a minimal right-resolving path set presentation of certain sofic shifts. Appendix B establishes a standard
equivalence between different definitions  of one-sided sofic shifts.

\subsection*{Acknowledgments} W. Abram received support from  an
NSF  Graduate Research Fellowship. 
J. C. Lagarias received  support from NSF grant  DMS-1101373.

%****************************************************************************************
%
%
%  Section 2: Examples
%
%
%****************************************************************************************

\section{Examples }\label{sec2}

We  present  three simple examples of path sets illustrating some 
of their  distinctive features.
The first two examples illustrate  features coming from
the imposition of an initial condition (the choice of  initial vertex $v$), 
showing how shift-invariance may fail.
The  third example shows that path
sets are not closed under the complement operation in a fixed alphabet.

% EXAMPLE 2.1
\begin{exmp}\label{ex21}
%%%%%%%%%%%%% 
%({\em Path sets that are not shift-invariant})
%%%%%%%%%%%%%
For the alphabet $\sA = \{0, 1, ..., n-1\}$ and a given symbol $j \in \sA$, 
consider the two-state graph $\sG$ pictured in Figure \ref{fig1}. 
It has $n$ distinct loops at vertex $v_1$, labelled $0, 1, ..., n-1$, respectively,
which are denoted schematically in the figure.
The directed graph $\sG$ is connected but not strongly
connected, i.e. it is reducible. 
%%%%%%%%%%%%%
%The edge labelled $0,1, \ldots, n-1$ is meant to denote $n$ separate edges, one with each label. 
%%%%%%%%%%%%%
The path set $X_{\sG}( v_0)$ is the prefix set
$\sZ_j = j \sA^{\mathbb{N}}$. This set  is not shift-invariant, but is of the form $W \overline{Y}$
where $W$ is a finite set, and $\overline{Y}$ is shift-invariant, being the full shift. 
We may  call 
path sets that are finite unions of sets of the  form $W {\overline{Y}}$  {\em prefix-shift-invariant}. 
The  sets  $\sZ_j$ are the simplest examples of path sets that are not  shift-invariant.
They  play an important role in the structure
Theorem \ref{chartheorem} for path sets. 
\end{exmp}

\begin{figure} 
	\centering
	\psset{unit=1pt}
	\begin{pspicture}(-80,0)(80,150)
		\newcommand{\noden}[2]{\node{#1}{#2}{n}}
		\noden{$v_0$}{-30,20}
		\noden{$v_1$}{30,20}
		\aline{n$v_0$}{n$v_1$}{j}
		\bcircle{n$v_1$}{270}{0,1,\ldots,n-1}
	\end{pspicture}
	\caption{A two-state presentation $(\mathcal{G}, v_0)$ of the path set $\sZ_1:= 1 \{ 0, 1\}^{\mathbb{N}}$}\label{fig1}
\end{figure}
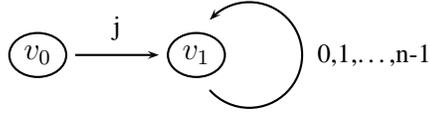

% EXAMPLE 2.2
\begin{exmp}\label{ex22}  
%({\em Path set that is not shift-invariant with prefixes })
In Example \ref{ex21} the path sets are given as a finite prefix followed by a
shift-invariant set. Now 
consider the three-state graph $\sG$ pictured in
Figure \ref{fig2}, which uses edge labels drawn from
$\sA:= \{ 0, 1, 2\}$. The graph $\sG$ is irreducible. 
 The path sets $\sP_i =X_{\sG}( v_i)$ for $0 \le i \le 2$ consist of 
a single infinite path:
$\sP_0 = \{(012)^{\infty}\}, ~\sP_1=\{(120)^{\infty}\},  \sP_2= \{(201)^{\infty}\}$.
These path sets are not shift-invariant.  This example  cannot be partitioned into 
a finite number of path sets of the type $W \overline{Y}$ of example \ref{ex21}, where $W$ is a finite word,
and $\overline{Y}$ is itself a shift-invariant path set.

\begin{figure} 
	\centering
	\psset{unit=1pt}
	\begin{pspicture}(-80,0)(80,150)
		\newcommand{\noden}[2]{\node{#1}{#2}{n}}
		\noden{$v_0$}{0,90}
		\noden{$v_1$}{-30,15}
		\noden{$v_2$}{30,15}
		\bline{n$v_0$}{n$v_1$}{0}
		\bline{n$v_1$}{n$v_2$}{1}
		\bline{n$v_2$}{n$v_0$}{2}
	\end{pspicture}
	\caption{$(\sG, v_i)$ gives path sets consisting of one element}\label{fig2}
\end{figure}
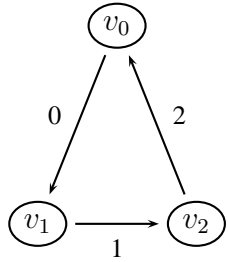

\end{exmp}
%EXAMPLE 2.3

\begin{exmp}\label{ex23}
({\em A path set whose complement is not a path set})
Consider the two state graph $\sG$ pictured in Figure \ref{fig3}
with edge labels drawn from  $\mathcal{A} = \{0,1\}$. 
Here the path set $\mathcal{P} = X_{\mathcal{G}}(v_0)$ associated to vertex $v_0$.
 is the set of all sequences in $\mathcal{A}^\mathbb{N}$ in which the block $11$ does not occur. 

We show by contradiction that the complement $\mathcal{P}^c := \sA^{\NN} \smallsetminus \mathcal{P}$ is not a path set.
Here $\mathcal{P}^c$
 is the set of 
all sequences in $\mathcal{A}^\mathbb{N}$ in which the sequence $11$ 
occurs at least once. 
Suppose  $\mathcal{P}^c$ were presented as a path set by a 
pointed graph $\mathcal{G}(v)$. Let $n = | \sG |$. 
Now $\mathcal{P}^c$ contains words with arbitrarily long initial substrings which do not contain the string $11$.
 It  follows that we can find an initial segment $a_0 a_1 \cdots a_k$ of a word in $\mathcal{P}^c$ subject to the following constraints:
\begin{enumerate}[(1)]
\item $a_0 a_1 \cdots a_k$   is given by an edge walk $e_0 e_1 \cdots e_k$ originating at the initial vertex $v$;
\item the string $11$ is not contained in $a_0a_1 \cdots a_k$.
\item $k>n$ and $a_k =0$,
\item the terminal state of edge $e_k$ is visited earlier as the initial state of some $e_j$, $0 \le j < k$.
\end{enumerate}
 It follows from (1)-(4) that 
$a_0a_1 \cdots a_{j-1} (a_j a_{j+1} \cdots a_k)^\infty \in \mathcal{P}^c.$
But this path does not contain $11$, so it is in $\mathcal{P}$, a contradiction.
Thus $\mathcal{P}^c$ cannot be a path set.

\begin{figure} 
	\centering
	\psset{unit=1pt}
	\begin{pspicture}(-80,0)(80,150)
		\newcommand{\noden}[2]{\node{#1}{#2}{n}}
		\noden{$v_0$}{-30,20}
		\noden{$v_1$}{30,20}
		\dline{n$v_0$}{n$v_1$}{1}{0}
		\bcircle{n$v_0$}{90}{0}
	\end{pspicture}
	\caption{The path set of $(\mathcal{G},v_0)$ has  complement
	% $\sA^{\NN} \smallsetminus \sP$ 
	that is not a path set} 
		\label{fig3}
\end{figure}
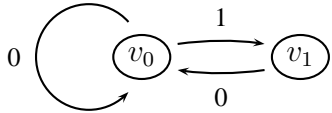

\end{exmp}

%****************************************************************
%
% Section 3: Presentations%
%****************************************************************

\section{Presentations of Path Sets}\label{sec3}

 A given path  set $\mathcal{P}$ arises from many different pointed graphs.
 If we let $\mathcal{G}'$ be the labeled graph obtained from $\mathcal{G}$ by keeping only 
 the connected component containing the vertex $v$ and iteratively throwing out any stranded states 
 (those vertices with no exit edges)
 and the edges into or out of them (except $v$), then clearly $X_\mathcal{G}(v) = X_{\mathcal{G}'}(v)$.
 Therefore, from a symbol space point of view, it suffices to consider graphs $\mathcal{G}$ such that each vertex of $\mathcal{G}$ is reachable from $v$ by a path in 
  $\mathcal{G}$, and no vertices have out-degree zero. We call such labelled graphs {\em pruned graphs}.
  (Compare \cite[Prop. 2.2.10]{LM95}).

Suppose $\mathcal{P}$ is a path set with presentation $(\mathcal{G},v)$. Given an initial string 
$a_0a_1a_2\cdots$ belonging to a member of $\mathcal{P}$, 
it would be useful to know that this string corresponds to a unique finite walk
 in $\mathcal{G}$ originating at 
the distinguished vertex $v$. 
This is not true for a general $\mathcal{G}$, but is true for a special class of graphs. 

%************************
% Defn 3.1
%************************

\begin{defn}\label{de31}
A finite labeled directed graph $\mathcal{G}$ is called \emph{right-resolving} (or {\em deterministic})
if for any vertex $w$ of $\mathcal{G}$, 
there is at most one edge originating at $w$ with any given label. A path set with a right-resolving presentation 
has the uniqueness property given above.
\end{defn} 

Recall that a presentation $(\sG, v)$ is termed {\em reachable} if each vertex of $\sG$ can be reached from $v$.
%************************
% theorem 3.2
%************************

\begin{thm} 
\label{rrthm} 
Every path set $\mathcal{P}$ has a right-resolving presentation that is reachable.
\end{thm}

\begin{proof}  It suffices to construct a  right-resolving presentation, since the reachability condition
is then achieved by discarding  all vertices of $\sG$ that are not reachable from $v$.

The existence of a right-resolving presentation $\sH$ for (two-sided) sofic shifts,
constructed directly from a given presentation $\sG$,  is shown 
   in Lind and Marcus (\cite{LM95}, Theorem 3.3.2, p. 76), using
   the  subset construction described below. In the following Claim,
 we show this  construction yields the desired result for path sets, as well.

The {\em subset construction} starting from a labelled graph $\sG$
produces a new labelled  graph $\sH$  having $2^{|V(\sG)|}-1$ vertices which are  marked 
by  the nonempty subsets $S$ of the vertex labels of $\sG$.  
For each edge label $a \in \sA$ and a vertex $S$ of $\sH$
we assign an exit edge with this label mapping to the vertex $S'$ whose marking is the union of all 
vertices $w \in V(\sG)$ such that there is an edge from some $v \in S$ to $w$ with label $a$,
unless $S'=\emptyset$, in which case we  assign no edge.
The  labelled graph $\sH$ is clearly right-resolving. 

Now let  $(\sG, v_0)$ be  a presentation of a path set $\sP$.
We may assume without loss of generality that this presentation is a pruned graph,
viewed from vertex $v_0$.

{\bf  Claim.}  {\em The right-resolving construction applied
to the underlying graph $\sG$ of a pruned pointed graph $(\sG, v_0)$ has the property that each  vertex $S \in V(\sH)$ has path set
\begin{equation}  
X_{\sH}(S) =  \bigcup_{v' \in S} X_{\sG}(v').
\end{equation}
}

To prove the claim, let $\sB_m(\sP)$ denote the set of vectors giving in order  the first $m$ symbols of each path in $\sP$.
It then suffices to show the equality  of 
the sets $\sB_m ( X_{\sH}(S))$
and $\cup_{v' \in S} \sB_m(X_{\sG}(v'))),$ 
for all $m \ge 1$.
We prove the inclusions of these two sets  in both directions by induction on $m$.
The base case $m=1$ is clear, using the fact for $\sG$ that all one-step paths
extend to infinite paths.
For the inclusion of $\cup_{v' \in S} \sB_m(X_{\sG}(v'))$ in $\sB_m ( X_{\sH}(S))$
given $w \in S$ it suffices to note that a path in $\sG$ starting from a vertex $w$, with a
given symbol  set lifts to a path in $\sH$ starting from $S_0$  and the same symbol sequence,
by definition of the edges of $\sH$.
 For the other inclusion, given  a start vertex $S= S_0$, and proceeding to follow a path to vertices $S_1, S_2, ... , S_n$
in $\sH$, and given the symbol sequence, the right-resolving property gives us unique
backtracking from $S_n$ to $S_0$. 
A key point is that, although we may not  know what state of $\sG$ in $S_n$ 
we have reached  at the $n$-th step,
from knowledge of the $(n+1)-$st step we are guaranteed there exists some
 state of $S_n$ giving this symbol that permits backtracking 
through a series of states $w_j \in S_j$ to a state $w_0$ in $S_0$, since all states of $S_n$ permit such backtracking. Now the sequence corresponds to a
path in $\sG$ starting from $w_0$ that traverses the states $w_j$ with the correct letters. It corresponds
to the initial part of an infinite path in $X_{\sG}(w_0)$, because by our hypothesis on $\sG$ all finite paths
extend to some infinite path. Thus it belongs to  $\cup_{v' \in S} \sB_m(X_{\sG}(v')))$.
This completes the induction step, proving the claim.\smallskip

We apply the Claim  to the pruned labelled graph $(\sG, v_0)$ presenting $\sP$.  The
Claim shows that the  pointed graph $(\sH, S_{v_0})$ 
with vertex $S_{v_0}:= \{ v_0\}$ is a
right-resolving presentation of $\mathcal{P}$, which is the desired result.
Note that we may  obtain a (possibly smaller) right resolving presentation of $\sP$, by taking the 
induced subgraph of $\sH$ obtained by restricting to the set of vertices reachable from $S_{v_0}$
by a directed path. \end{proof}

The proof of Theorem~\ref{rrthm}  shows that any one-sided sofic shift $Y$ is a path set.
For if  $\sG$ represents such a shift, then
the Claim above applied to $\sG$ implies  that the pointed graph $(\sH, S_{\sG})$
 with  $S_{\sG} := V(\sG)$ represents the   path set $X_{\sH}(S_{\sG})$,
which is the one-sided sofic shift represented by the  graph $\sG$.  
We illustrate this on the following example.

%EXAMPLE 3.3
\begin{exmp}\label{ex33}
Consider the three-state graph $\sG$,
pictured in Figure \ref{fig4}, which uses the label alphabet $\sA=\{ 0, 1, 2\}$.
Here $\sG$ is  a presentation of a one-sided sofic shift $\sS$, 
but this presentation is  not right-resolving.
The subset construction results in the graph $\sH$ pictured in Figure \ref{fig5}. The graph 
$\sH$
has $7$ vertices, labelled by unions of states in $\sG$, and which is 
right-resolving.  The pointed graph $(\sH, ABC)$ gives 
a path set presentation of the sofic shift $\sS$. In this example the vertices labelled
$A, B, C$ can be pruned from this pointed graph, resulting in a $4$-state 
graph $\sH'$ such that the pointed graph $(\sH', ABC)$ gives also a path set representation
of the one-sided sofic shift $\sS$.
\end{exmp}

\begin{figure} 
	\centering
	\psset{unit=1pt}
	\begin{pspicture}(-80,0)(80,150)
		\newcommand{\noden}[2]{\node{#1}{#2}{n}}
		\noden{A}{0,100}
		\noden{B}{-50,15}
		\noden{C}{50,15}
		\dline{nA}{nC}{0}{2}
		\dline{nA}{nB}{0}{1}
		\dline{nB}{nC}{1}{2}
	\end{pspicture}
	\caption{A graph $\sG$ that is not right-resolving }\label{fig4}
\end{figure}
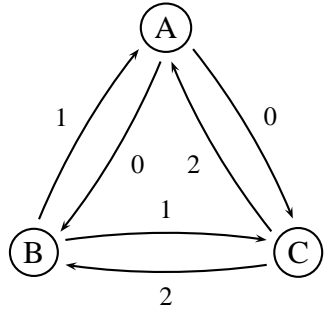

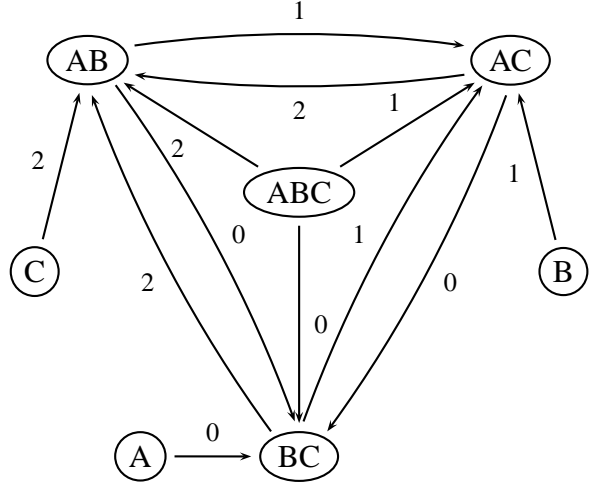
\begin{figure} 
	\centering
	\psset{unit=1pt}
	\begin{pspicture}(-100,0)(100,170)
		\newcommand{\noden}[2]{\node{#1}{#2}{n}}
		\noden{ABC}{0,100}
		\noden{AB}{-80,150}
		\noden{AC}{80,150}
		\noden{BC}{0,0}
		\noden{A}{-60,0}
		\noden{B}{100,70}
		\noden{C}{-100,70}
		\aline{nABC}{nAC}{1}
		\aline{nABC}{nBC}{0}
		\aline{nABC}{nAB}{2}
		\aline{nC}{nAB}{2}
		\aline{nA}{nBC}{0}
		\aline{nB}{nAC}{1}
		\dline{nAC}{nAB}{2}{1}
		\dline{nAC}{nBC}{0}{1}
		\dline{nAB}{nBC}{0}{2}
	\end{pspicture}
	\caption{Right-resolving graph $\sH$ obtained from $\sG$ by subset construction}\label{fig5}
\end{figure}

Returning to  the subset  construction, the number of vertices  
$|V(\sH)|=2^{|V(\sG)|} -1$,
which is exponentially larger than $|V(\sG)|$.
The following example shows this exponential blowup is  sometimes unavoidable 
using the subset construction to construct a path set realization of a sofic shift $Y$,
even when $\sG$ is a minimal right-resolving representation of $Y$. 
More is true: in this example the pointed 
graph $(\sH, V(\sG))$ produced by 
the subset construction can be shown to give  a minimal right-resolving path set presentation of $Y$,
see Appendix A.
Such a  result establishes that  
minimal right-resolving realizations of a one-sided sofic shift $\sS$ 
by a graph $\sG$ can sometimes be exponentially more efficient
than minimal right-resolving presentations as a path set.

%****************************
% Example 3.4: minimal right resolving
%****************************

\begin{exmp}\label{ex34}
%(Exponential Blow-Up  of States in Path Set Construction). 
Let $\mathcal{G}:=\sG_n$ have vertex set $V = \{v_0,\ldots,v_{n-1}\}$,
and have  edges  labelled 
from the  label set $\sL=\{ 0, 1, ..., 2n-1\}$, as follows. Each vertex $v_i$ has $n-1$ self-loops,
labelled with $n-1$ of the symbols from $(n, n+1, ..., 2n-1)$, omitting only the symbol $n+i$.
Between distinct $v_i$ and $v_j$ there is a directed edge labelled $ j-i ~(\bmod \, n)$, taking the
least nonnegative residue,  noting that  this label  necessarily falls in 
the interval $[1, n-1]$.  Figure ~\ref{fig6} pictures $\mathcal{G}_3$.
The  graph  $\sG$   
 is right resolving and irreducible (i.e. strongly connected). 
and has has $n(2n-2)$ directed edges.
 It is therefore an irreducible 
 right-resolving presentation of a one-sided sofic shift $Y= X(\sG)$.

 Applying the  subset construction  to $\sG$  yields a pointed graph $(\sH, S_{\sG})$
 giving a right-resolving presentation of $Y= X_{\sH}(S_{\sG})$ as a path set. The graph $\sH$
 has $2^n -1$ vertices.  We can check that the  vertex $S_{\sG}$ of $\sH$ has directed
 paths connecting it to every one of the vertices of $\sH$. 
 Namely, for $1 \le j \le n$ there is a directed  edge labelled
 $n+j$ connecting it to the vertex labelled $S_{\sG'}$ with $\sG' :=\sG \smallsetminus \{ v_j\}$.
 In a similar fashion for $i \ne j$  there is a directed edge labelled $n+i$ connecting  the vertex $S_{\sG'}$ to 
 $S_{\sG''}$ with $\sG'' :=\sG' \smallsetminus \{ v_i\}= \sG \smallsetminus \{ i, j \}$, and so on.
 Furthermore all the vertices $S_{v_i}$ for $1 \le i \le n$ form a strongly connected 
 component of this graph, reproducing the graph $\sG$, and this graph  can be reached  by a path
 from every other vertex in $\sH$. Thus there are no stranded states, so $(\sH, S_{\sG})$
 is a pruned pointed graph having $2^{n} -1$ vertices.
 The graph $\sH$ actually provides a minimal right resolving presentation of $Y$ as a
 path set, which we establish in  Appendix A.
  \end{exmp}

\begin{figure} 
	\centering
	\psset{unit=1pt}
	\begin{pspicture}(-100,0)(100,170)
		\newcommand{\noden}[2]{\node{#1}{#2}{n}}
		\noden{$v_0$}{0,110}
		\noden{$v_1$}{-50,40}
		\noden{$v_2$}{50,40}
		\dline{n$v_0$}{n$v_1$}{1}{2}
		\dline{n$v_0$}{n$v_2$}{2}{1}
		\dline{n$v_1$}{n$v_2$}{1}{2}
		\bcircle{n$v_0$}{0}{4,5}
		\bcircle{n$v_1$}{120}{3,5}
		\bcircle{n$v_2$}{240}{3,4}
	\end{pspicture}
	\caption{Sofic-shift presentation $\sG_3$ that exhibits exponential blow-up in its minimal right-resolving path set presentation
	$(\sH, V(\sG_3))$}\label{fig6}
\end{figure}
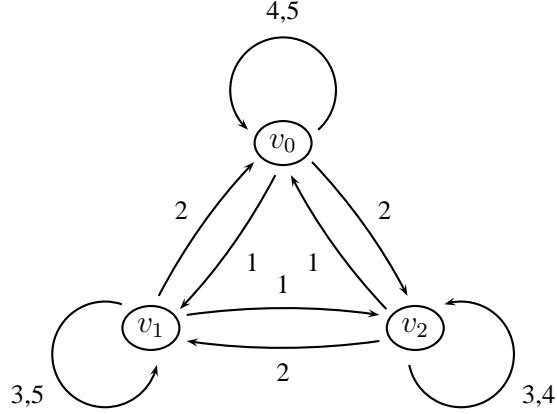

%****************************************************************
%
% Section 4/3A: Set-Theoretic Properties
%
%****************************************************************

\section{Set-Theoretic  Properties of Path Sets}\label{sec3b}

We  establish that  the collection of path sets are closed  under union and intersection, obtaining Theorem \ref{topthm}.
This collection of sets is not closed under complement by Example \ref{ex23}.
To establish closure of the collection of path sets under intersections,
we  will need an appropriate notion of graph product. Lind and Marcus \cite{LM95} make the following construction:
%************************
% Defn 4.1
%************************

\begin{defn}\label{de32}
 Let $\mathcal{G}_1 = (V_1,\mathcal{E}_1)$ and $\mathcal{G}_2=(V_2,\mathcal{E}_2)$ be directed
 labeled graphs 
 over the same alphabet $\mathcal{A}$. The \emph{label product}  of these graphs, written 
 \[
 \mathcal{G}_1 \star \mathcal{G}_2: = ( {V}_1 \times {V}_2, \mathcal{E})
 \]
 is the labeled directed graph with vertex set 
 ${V}_1 \times {V}_2$ and edge set 
 $$
 \mathcal{E}:= \{((e_1,e_2),\ell) \in \mathcal{E}_1 \times \mathcal{E}_2 \, | \,\text{$e_1$ and $e_2$ share the} \newline
 \text{ same label} ~\ell \in \sA\}.
 $$
  The label assigned
 to  the edge $(e_1,e_2)$ is the common label of the edges $e_1 \l$ and $e_2$.
\end{defn}

We  adapt this definition 
to get a \emph{pointed label product} for our pointed graphs 
$(\mathcal{G}_1,v_1)$ and $(\mathcal{G}_2,v_2)$ by  choosing the pair $(v_1,v_2)$ as the distinguished vertex  of 
$(\mathcal{G}_1 \star \mathcal{G}_2, (v_1, v_2))$, i.e., 
\[
(\mathcal{G}_1,v_1) \star (\mathcal{G}_2,v_2) := (\mathcal{G}_1 \star \mathcal{G}_2,(v_1,v_2)).
\]

 It is then possible to eliminate extraneous components and prune stranded states from the pointed label product 
to recover an equivalent pointed graph $(\mathcal{G},v)$ that is irreducible and such that each vertex of $\mathcal{G}$ 
is reachable from the distinguished vertex.

%************************
% Proposition 4.2
%************************

\begin{lem} If $\mathcal{G}_1$ and $\mathcal{G}_2$ are  labeled directed graphs
that are right-resolving, then the  label product graph   $\sG := \mathcal{G}_1 \star \mathcal{G}_2$ is also right-resolving.
\end{lem}

\begin{proof} 
This is shown in Lind and Marcus (\cite{LM95}, Proposition 3.4.10, p. 89).
\end{proof}

This result applies to the pointed label product, since the right-resolving property  is determined by the underlying labelled graph.
Just as the label product interacts nicely with the structure of sofic shifts (see Lind and Marcus \cite[Prop. 3.4.10, p. 89]{LM95} )
we have the following result  concerning the pointed label product of pointed graphs and their associated path sets.
%************************
% Proposition 4.3
%************************

\begin{prop} \label{prodprop}
 If $(\mathcal{G}_1,v_1)$ and $(\mathcal{G}_2,v_2)$ are pointed graphs over the same alphabet $\mathcal{A}$, then
\begin{equation} 
X_{\mathcal{G}_1}(v_1) \cap X_{\mathcal{G}_2}(v_2) = X_{\mathcal{G}_1 \star \mathcal{G}_2}((v_1,v_2)).
\end{equation}
\end{prop}

\begin{proof} Suppose
 $(\mathcal{G}_1,v_1)$ and $(\mathcal{G}_2,v_2)$ are pointed graphs. Let 
 $\mathcal{G} = \mathcal{G}_1 \star \mathcal{G}_2$ be the label product, and let $v=(v_1,v_2)$, a node of $\mathcal{G}$. 
 Then $(\mathcal{G},v)$ is also a pointed graph. 
 
 We prove now that $X_\mathcal{G}(v) = X_{\mathcal{G}_1}(v_1) \cap X_{\mathcal{G}_2}(v_2)$. 
Suppose 
$$(e_{11},e_{21}),(e_{12},e_{22}),(e_{13},e_{23}),\ldots$$
 is an infinite walk in 
$\mathcal{G}$ originating at $v$. Then by the definition of the label product $e_{11},e_{12},e_{13},\ldots$ is an infinite walk in $\mathcal{G}_1$ originating at $v_1$ and $e_{21},e_{22},e_{23},\ldots$ is an infinite walk in 
$\mathcal{G}_2$ originating at $v_2$. The label of $(e_{11},e_{21}),(e_{12},e_{22}),(e_{13},e_{23}),\ldots$
 is equal to the label of $e_{11},e_{12},e_{13},\ldots$ and of $e_{21},e_{22},e_{23},\ldots$ by definition, so 
\[X_\mathcal{G}(v) \subseteq X_{\mathcal{G}_1}(v_1) \cap X_{\mathcal{G}_2}(v_2).
\]
Conversely, suppose $a_1a_2a_3 \ldots \in X_{\mathcal{G}_1}(v_1) \cap X_{\mathcal{G}_2}(v_2)$. 
Then there is an infinite edge walk $e_{11},e_{12},e_{13},\ldots$ in $\mathcal{G}_1$ originating at $v_1$ 
and an edge path $e_{21},e_{22},e_{23},\ldots$ in $\mathcal{G}_2$ originating at $v_2$ such that both of these paths have label $a_1a_2a_3 \ldots$. Since $e_{1i}$ and $e_{2i}$ share the same label for all $i \in \mathbb{N}$, there is path $(e_{11},e_{21}),(e_{12},e_{22}),(e_{13},e_{23}),\ldots$ in $\mathcal{G}$ originating at $v=(v_1,v_2)$ which also has label $a_1a_2a_3 \ldots$. Thus, we also have
\[X_{\mathcal{G}_1}(v_1) \cap X_{\mathcal{G}_2}(v_2) \subseteq X_\mathcal{G},
\]
and the proposition follows.
\end{proof}

%************************
% Proof of Theorem 1.2
%************************

\begin{proof}[Proof of Theorem~\ref{topthm}.]  
Let $\mathcal{P}$ be a path set in $\sC(\sA)$.
\newline (1) It is immediate from the definition that $\mathcal{P}$ is a closed subset of $\mathcal{A}^\mathbb{N}$.
\newline (2) If follows immediately from Proposition ~\ref{prodprop} that the collection of path sets in $\mathcal{A}^\mathbb{N}$ is closed under intersections.
\newline (3) Suppose $\mathcal{P}_1$ and $\mathcal{P}_2$ are path sets in $\mathcal{A}^\mathbb{N}$.
 We must show that $\mathcal{P}_1 \cup \mathcal{P}_2$ is a path set. Choose presentations
  $(\mathcal{G}_1,v_1)$ and $(\mathcal{G}_2,v_2)$, and construct $(\mathcal{G},v)$ as follows: 
  Let $\mathcal{G}$ be obtained from $\mathcal{G}_1 \amalg \mathcal{G}_2$ by adding a new vertex $v$, 
  and, for each edge $v_i \rightarrow w$ in $\mathcal{G}_i$ labeled $a$, 
  an edge $v \rightarrow w$ in $\mathcal{G}$ labeled $a$. 
  We now must show that $X_\mathcal{G}(v) = \mathcal{P}_1 \cup \mathcal{P}_2$.

Indeed, suppose $a_0a_1a_2 \cdots \in \mathcal{P}_1 \cup \mathcal{P}_2$. 
Then without loss of generality $a_0a_1a_2 \cdots \in \mathcal{P}_1$. 
Let $v_1 \rightarrow w$ be the first edge in an infinite walk in $\mathcal{G}_1$ with edge-label 
$a_0a_1a_2 \cdots$. Then by assumption there is an edge $v \rightarrow w$ labeled $a_0$, 
and replacing the edge walk $v_1 \rightarrow w$ with $v \rightarrow w$ gives an infinite walk 
in $\mathcal{G}$ originating at $v$ with edge-label $a_0a_1a_2\cdots$. 
Thus, $\mathcal{P}_1 \cup \mathcal{P}_2 \subset X_\mathcal{G}(v)$.

Conversely, suppose $a_0a_1a_2\cdots \in X_\mathcal{G}(v)$. 
Let $v \rightarrow w \rightarrow \cdots$ be an infinite walk in $\mathcal{G}$ with edge label 
$a_0a_1a_2 \cdots$. By construction $\mathcal{G}$ has no self-loops at $v$, 
so without loss of generality $w \in \mathcal{P}_1$, and in fact there is an edge
 $v_1 \rightarrow w$ labeled $a_0$. But there are no edges connecting the $\mathcal{G}_1$
  component of $\mathcal{G}$ to the $\mathcal{G}_2$ component of $\mathcal{G}$, 
  so it must be that the remaining edges of our walk remain in $\mathcal{G}_1$. 
  Thus, replacing $v \rightarrow w$ by $v_1 \rightarrow w$ gives an infinite walk in $\mathcal{G}_1$
   originating at $v_1$ with edge-label $a_0a_1a_2 \cdots$. 
   Therefore $X_\mathcal{G}(v) \subset \mathcal{P}_1 \cup \mathcal{P}_2$, 
   hence $X_\mathcal{G}(v) = \mathcal{P}_1 \cup \mathcal{P}_2$, 
   and we conclude that $\mathcal{P}_1 \cup \mathcal{P}_2$ is a path set. This proves the theorem.
\end{proof}

%****************************************************************
%
% Section 5. Relation to  Sofic Shifts
%
%****************************************************************

\section{Path Sets and  One-Sided Sofic Shifts}\label{sec4}

%************************
% Proof of Theorem 1.4
%************************

\begin{proof}[Proof of Theorem~\ref{softhm}] 
(1) Let  the path set $\sP= X_{\sG}(v_0)$. The  shifted set $\sigma(\sP)$ is a finite union of 
the path sets $X_{\sG}(v_i)$, where 
$v_i$ runs over the vertices reachable from $v_0$ in one step. Thus $\sigma(\sP)$ is a path set
using closure of  path sets under unions (Theorem \ref{topthm}(3)).
%Theorem ~\ref{th11}(2). 

\noindent (2) Let $\mathcal{P}$ be a path set, 
and let $\overline{\mathcal{P}}$ be its one-sided shift closure. Suppose $\mathcal{G}(v)$ is a 
presentation of $\mathcal{P}$. By the discussion in the introduction we can assume that every vertex of 
$\mathcal{G}$ is reachable in $\mathcal{G}$ by a path originating at $v$. Then
\begin{equation} \sigma(\mathcal{P}) = \cup_{\{w \in \mathcal{G} : \exists \text{ an edge } v \rightarrow w\}} X_\mathcal{G}(w),
\end{equation}
and hence 
\begin{equation} \label{closureeq} 
\overline{\mathcal{P}} = \cup_{j \in \mathbb{N}} \sigma^j (\mathcal{P}) = \cup_{w \in \mathcal{G}} X_\mathcal{G}(w).
\end{equation}
Since $\mathcal{G}$ is a finite graph, this says that $\overline{\mathcal{P}}$ is a finite union of path sets, hence is a path set by Theorem ~\ref{topthm}. 
 \medskip

\noindent (3) For a path set $\sP$ represented by a pruned pointed graph $(\sG, v_0)$, 
  its shift-closure $\overline{\mathcal{P}}$ is exactly the one-sided sofic shift presented by $\mathcal{G}$,
 since every vertex of $\sG$ is reachable from $v_0$, and $\overline{\mathcal{P}}$ is a path set by (1).
 Thus if $\mathcal{P}$ is shift-invariant, i.e. $\mathcal{P} = \overline{\mathcal{P}}$, 
then $\mathcal{P}$ is a one-sided sofic shift.

For the converse, 
suppose $Y \subset \mathcal{A}^\mathbb{N}$ is a one-sided sofic shift. 
Then $Y$ is the union of the path sets at all the vertices of $\sG$. But the collection of path sets
is closed under set union by Theorem \ref{topthm} (3) 
whence $Y$ is a path set. (An alternative proof is obtainable by direct construction via the subset construction,
as indicated  after Theorem \ref{rrthm}. )
\end{proof}

 %***************************************************************************************************
%
% Section 6: 
%
%*****************************************************************************************************

\section{Path Sets and Two-Sided Sofic Shifts}  
\label{sec4a}

We  relate path sets to the infinite follower sets defined in Section \ref{sec12} for two-sided sofic shifts.
Recall that  for a two-sided sofic shift $Y$ and a finite word $\bw:=\alpha_{-k}\alpha_{-k+1}\cdots   \alpha_{-2}\alpha_{-1}$,
the   {\em infinite follower set} $\sF_{Y}(\bw) \subset \sA^{\mathbb{N}}$ 
is  the set of all possible one-sided (infinite) paths $\alpha_0\alpha_1 \alpha_2...$ in $Y$ that can directly follow 
occurrences of the word $\bw$ in the  sofic shift.

%************************
%  Theorem 6.1 (formerly th41a)
%***********************
\begin{thm}\label{th61} {\em (Path sets and infinite follower sets)}

(1) Each infinite follower set $\sF_{Y}(\bw) \subset \sA^{\mathbb{N}}$ coincides with  a path set in the same alphabet $\sA$.

(2) Not all path sets $\sP \subset \sA^{\NN}$ are infinite follower sets in the same alphabet $\sA$. 
However any path set $\sP$  is representable as an infinite follower set using an enlarged alphabet $\sA' := \sA \cup \{ a'\}$ containing
 one extra symbol.
\end{thm}

\begin{proof}

(1)  Two-sided sofic shifts are characterized among all two-sided subshifts in $\sA^{\mathbb{Z}}$
by the property of  having a finite number of distinct infinite (right-) follower sets (Fischer \cite[Theorem 2]{Fis75b}).
%Also may \cite{Kr83}).
Given a  two-sided 
sofic shift $Y \subset \sA^{\mathbb{Z}}$,  its {\em right Krieger cover} (\cite{Kr83}) is constructed by taking a labelled
graph $\sG$ whose states have  path sets corresponding to  all the possible  infinite follower sets $\sF_{Y}(\bw)$  of $Y$.
By construction the path sets in the right Krieger cover 
 coincide  with  these infinite follower sets. This shows that    all infinite follower sets
  are path sets in the same alphabet.

(2) We claim  that the prefix  sets $\sZ_j \subset \sA^{\mathbb{N}} $  with $\sA = \{0,1 ,..., n-1\}$ in Example \ref{ex21},
with $n \ge 2$, are path sets that are  {\em not} infinite follower sets in
any two-sided sofic shift $Y$ using the same alphabet $\sA$. For simplicity we show this for 
the prefix set  $\sZ_1= 1 \{0,1\}^{\mathbb{N}}$ 
pictured in Figure \ref{fig1}. 
In this alphabet any two-sided sofic system in 
which $\sZ_1$ is included must contain its  two-sided shift closure,
and  this is  the full two-sided shift $Y:=\{0, 1\}^{\mathbb{Z}}$. 
However any  infinite follower set $\sF_{Y}( \alpha_{-k-1} \alpha_{-k} \cdots \alpha_{-1})$
for the full shift $Y$  will itself be the full one-sided
shift $\{0,1\}^{\NN}$, and this never equals $\sZ_1$. 
That is,  the ``initial condition" data in $\sZ_1$ is lost when taking  the shift closure. Thus $\sZ_1$ is not 
an infinite follower set in any two-sided sofic system using the alphabet $\sA=\{0,1\}$.

It is easy to see that any path set $\sP \subset \sA$ can be represented as an infinite follower set in an 
enlarged alphabet $\sA'= \sA \cup \{ a'\}$ adding one extra symbol to the label set, call the extra symbol $n$. 
 The extra symbol $``n"$ will supply the ``initial condition" data determining the path set.
Take a presentation
$(\sG, v_0)$ for $\sP$, and create a new graph $\sG'$ which
consists of  the graph $\sG$  augmented with one extra vertex $v^{*}$, which has two edges, a self-loop
labelled with a symbol from $\sA$, call it $``0"$, and a directed  edge from $v^{*}$ to $v_0$ labelled $``n"$.
We let $Y$ denote the two-sided sofic shift associated to $\sG'$.  The  shift $Y$ certainly contains all two-sided infinite paths
of the form $0^{\infty}\, n \, \bs$, where $\bs \in \sP$.
Now $\bw= \alpha_{-1} = n.$ We assert that the infinite follower set $ \sF_{\sG'}( \bw)=\sP$. 
This is clear, since there is only one edge labelled $``n"$, which uniquely identifies the vertex $v_0$. 
\end{proof}

We  make a few remarks comparing  minimal presentations of these structures.

\begin{defn}\label{de41b}
 Call a   presentation of  either a two-sided sofic shift or a one-sided sofic shift  or  a path set   {\em minimal} 
 if it minimizes the number of vertices in the associated labelled graph $\sG$, over all presentations of
the same type. 
\end{defn} 

\noindent{\bf Remarks.}
(1) For two-sided sofic shifts, one can always find a presentation in which
the underlying labelled graph $\sG$ has a property called being ``trim"
(in the sense of finite automata).
A  labelled directed graph is {\em trim}
if  every finite path in it can be extended to some doubly infinite path.
(For related but not identical concepts of ``trim" see  Willems \cite[Sec. 1.4.5]{Wil89}
and  Eilenberg \cite[p. 23]{Eil74A}.) Minimal presentations of two-sided
sofic shifts are necessarily trim.  In contrast, some path sets
have no presentation $(\sG, v)$ in which $\sG$ is trim;
 the basic sets $\sZ_j$  provide an example, 
 as follows from the proof of  (2) above in Theorem \ref{th61}.

(2) There has been much study of minimal presentations of two-sided sofic shifts.
If a two-sided sofic shift is  transitive (i.e. it has a presentation whose graph is strongly connected)
 then all minimal presentations
are isomorphic (\cite[Theorem 4]{Fis75b}. 
For reducible two-sided sofic shifts it is known that there can be non-isomorphic minimal presentations. 
For irreducible two-sided sofic shifts $Y$
 the right Krieger cover always gives  a minimal right-resolving
 presentation as a two-sided sofic shift.\footnote{This fact can be shown by observing that by construction it 
 has the {\em follower-separated property} (i.e., every graph vertex has a different follower set);
 and that its graph is irreducible (strongly connected), which is a consequence of the
 irreducibility of $Y$. Then one can use the fact that  the infinite follower-separated property 
  implies the finite follower-separated property, to show minimality,
  by invoking  \cite[Corollary 3.3.19]{LM95}.}

    (2)  A variant of the right Krieger  cover construction can be made for  
  path sets.
  Given a path set $\sP$, this   construction considers the infinite follower sets 
 $\sP(\bw)$  consisting of all paths that can follow 
 a given  finite {\em initial} word $\bw := \beta_1 \beta_2 \cdots \beta_k$ of some path in  $\sP$.
viewing it as corresponding to $ \alpha_{-k}\cdots  \alpha_{-3} \alpha_{-2}\alpha_{-1}$.
Here the empty word $\bw= \emptyset$ is allowed.
 As $\bw$ varies over all finite words,  only a finite number of distinct sets $\sP(\bw)$ are obtained this way,
 by virtue of the Structure Theorem \ref{chartheorem}.
 We then create a graph $(\sH, v_0)$
 whose vertices correspond to all the distinct $\sP(\bw)$,
 with the marked vertex $v_0$ corresponding to $\sP$,
 and with edges added in the obvious fashion (a directed edge labelled $j$
 is added from $\sP(\bw)$ to $\sP(\bw,j)$, checking this is well-defined).
    This variant construction  gives  a right-resolving presentation of the path set $\sP$.
   However we do not know   if it is always true in the one-sided case that when $\sP$ is irreducible, this
   construction gives a minimal path set realization.     The proof of minimality 
  in the two-sided case   does not  extend 
  because in  the one-sided case  the resulting graph $(\sH, v_0)$ may be reducible.
  Example \ref{ex33}  gives   examples  in which $\sH$ is reducible, but we
  are able to prove 
  minimality of  path set presentations in these cases
  by the specific argument in Appendix A.

%****************************************************************
%
% Section 7
%****************************************************************

\section{Decimation of  Path Sets}\label{sec5a}

\begin{proof}[Proof of Theorem \ref{fracthm}.]
(1) It suffices to prove the result when  $j=0$. 
For we have 
$$
\sP_{j, m} = \psi_{j, m} (\sP) = \psi_{0, m}(\sigma^j(\sP)).
$$
Now $\sP' := \sigma^j (\sP)$ is a path set by Theorem \ref{softhm}(1),
so to establish that  $\sP_{j, m}$ is a path set, it suffices to show that $\psi_{0,m}(\sP')$ is a path set.

For the case $j=0$,  we suppose given   a right-resolving presentation $(\sG, v)$ of $\sP$.
We derive from it  a presentation $(\sG', v)$ of $\sP_{0, m}$, which is generally not right-resolving,
as follows. We make an initial labelled graph $\sG_1$ which contains vertex $v$ and
all nodes reachable in one step from $v$, which in the case of self-loops from $v$ is
altered to add  a new copy of 
node $v$, labelled $v^{(1)}$, to which all self-loops of $v$ are required to exit. 
Secondly, from each new node $w$ reached after one step, one creates a (large) set of 
new vertices, corresponding to each $m$-step directed path leaving from that node $w$,
with vector of labels $\ell^{'} :=(\ell_1, \ell_2, ..., \ell_m)$, say. 
If $w'$ is the identifier of the vertex reached in $\sG$
along this path, then the  new vertex created is  identified by the symbol  $(\ell^{'}, w')$ 
There is now assigned a single directed edge from $w$ to $w'$ with label $\ell_m$.
Finally we add edges between $(\ell^{'}, w')$ and $(\ell^{''}, w'')$ whenever there is an
$m$-step path in $\sG$ from vertex $w'$ to $w''$ described by the vector $\ell^{''}$, and assigned
label $\ell_m^{''}$ the last step in that path. 
This describes the full graph $\sG'$, and the initial vertex remains $v$, as before.

It is straightforward to check that the paths from $v$ in $\sG''$ produce all paths in $\sP_{0,m}$,
and that each such path occurs at least once.

(2) Suppose that $\sP$ is invariant under the one-sided shift. Now we have
$$\sP_{j,m} = \psi_{0,m}(\sigma^j (\sP)) = \psi_{0, m}(\sP) = \sP_{0, m},$$
so all are equal, to a set denoted $\sP_m$, say. Finally $\sigma (\sP_{j, m} )= \sP_{j+m, m} = \sP_{j, m},$
showing that $\sP_{j,m}$  is shift-invariant.
\end{proof}

 %***************************************************************************************************
%
% Section 8: A Characterization
%
%*****************************************************************************************************

\section{Structural Characterization of Path Sets}\label{sec6}

\begin{proof}[Proof of Theorem \ref{chartheorem}]
(1) $\implies$ (2).  Let a path set $\sP \in \sC(\sA)$ be given.
We construct a finite set $\mathcal{S}_0$ closed under the two
operations of (2), as follows.
Take a right-resolving presentation $\mathcal{G}$ of the path-set $\sP$,
as given by Theorem \ref{rrthm}.
Let $V(G) = \{ v_k\} $ denote the 
set of vertices of $\mathcal{G}$ and let $\sP_k := X_{\mathcal{G}}(v_k)$ denote  the corresponding
path set, with $v_0$ assigned the  path set $\sP_0=\sP$. 
Now the collection 
$$
\mathcal{S}_0 := \{ \bigcup_{k \in J} \sP_k, \,\Big(\bigcup_{k \in J} \sP_k\Big)  
\cap \sZ_j: \mbox{all} \, J \subset V(\mathcal{G}) , \mbox{all}\, j \in \mathcal{A}\}
$$
 is a finite collection
of closed subsets of  $\mathcal{A}^\mathbb{N}$  that contains
$\sP$. It is manifestly closed under intersection with the prefix sets $\sZ_j$, since the
 sets $\sZ_j$'s are pairwise disjoint. Furthermore since $\sigma(X_k)$ is a finite
union of the other $\sP_{k'}$, and since each 
  $\sigma(\sP_k \cap \sZ_j)$ is a single $\sP_{k'}$ or the empty set (using the right-resolving property),
  the set $\mathcal{S}_0$ is also closed under the forward shift operation.
Thus (2) follows.

(2) $\implies$ (1).  Let $\mathcal{S}$ be
the smallest collection of sets containing $\sP$ and  closed under intersections with all the
prefix sets $\sZ_j$
and under the one-sided shift.  By hypothesis there exists  a finite set $\mathcal{S}_0$ containing $\sP$
and closed under these operations, and it  contains $\mathcal{S}$, therefore $\sS$ is  finite.
We construct a right-resolving presentation for $\sP$ using $\sS$
as follows. 
Form a graph whose nodes
are labelled with all sets is $\sS$. 
Let $\sP_{0,j_1} = \sP \cap \sZ_{j_1}$, and let $\sP_{1, j}= \sigma( \sP_{0, j})$
and put an edge labelled $j$ from $\sP$ to the node labelled $X_{1, j}$. Form now
$\sP_{1, j_1, j_2} = \sP_{1, j_1} \cap \sZ_{j_2}$, and make an edge labelled$j_2$
to a node labelled by the set $\sP_{1, j_1, j_2}$, and so on. A finite graph results,
which is right-resolving by construction. It contains $\sP$ as a label of   one node.
Now one may check that the definitions imply that the set of labelled paths through
this graph, starting at $\sP$, correspond exactly to the members of $\sP$. This
certifies that $\sP$ is  a path set. 
\end{proof}

%****************************************************************
%
% Section 8: Entropy
%
%****************************************************************

\section{Topological Entropy}\label{sec5}

%************************
% Proof of Theorem 1.8
%************************

\begin{proof}[Proof of Theorem~\ref{entthm}]
 Let $\mathcal{P}$ be a path set with right-resolving presentation 
 $\mathcal{G}(v)$, so that each vertex of $\mathcal{G}$ is reachable from $v$. 
 Thus $\overline{\mathcal{P}}$ is a one-sided sofic shift with presentation 
 $\mathcal{G}$. 
  The equality $H_{top}(\mathcal{P}) = H_{top}(\overline{\mathcal{P}})$
  follows directly from  the definitions, because $\mathcal{P}$ and $\overline{\mathcal{P}}$
  have the same set of internal blocks, so $N_n(\mathcal{P}) = N_n(\overline{\mathcal{P}})$.

  It remains to 
   show $H_p(\mathcal{P}) = H_{top}(\overline{\mathcal{P}})$.
    For this it  suffices 
   to show that $N^I_n(\mathcal{P})$ is asymptotic to $N_n(\overline{\mathcal{P}})$. 
Clearly, $N^I_n(\mathcal{P}) \leq N_n(\overline{\mathcal{P}})$. Now we will 
show that there is a fixed $k$ such that $N^I_{n+k}(\mathcal{P}) \geq N_n(\overline{\mathcal{P}})$
 for all $n$. $\mathcal{G}$ is by definition a finite graph, and by assumption each vertex 
 of $\mathcal{G}$ is reachable from $v$. Thus, there is a finite integer $k$ such that 
 each vertex of $\mathcal{G}$ is reachable by a path originating at $v$ of length no longer than $k$.
  For a vertex $w$ of $\mathcal{G}$ let $P_\mathcal{G}(w,n)$ denote the set of all paths of length 
  $n$ in $\mathcal{G}$ originating at $w$. Then
\begin{equation} \label{Eqent1} 
N^I_n(\mathcal{P}) = | P_\mathcal{G}(v,n)|
\end{equation}
and 
\begin{equation} \label{Eqent2} 
N_n(\overline{\mathcal{P}}) = \Big| \bigcup_{w \text{ a vertex of } \mathcal{G}} P_\mathcal{G}(w,n)\Big|.
\end{equation}

We construct an injective maps $f_w:P_\mathcal{G}(w,n) \rightarrow  P_\mathcal{G}(v,n+k_w)$ for each vertex 
$w$ of $\mathcal{G}$, where $k_w$ is an integer less than or equal to $k$, such that the images of the maps $f_w$ are non-intersecting. 

For a vertex $w \in \mathcal{G}$, choose some path $e_1e_2\ldots e_{r_w}$ from $v$ to $w$ for some $r \leq k$. Then for a path $e_1'e_2'\ldots e_n'$ of length $n$ originating $w$, let 
\begin{equation} f_w (e_1' e_2' \ldots e_n') = e_1 e_2 \ldots e_{r_w} e_1' e_2' \ldots e_n' \in P_\mathcal{G}(v,n+r) 
\end{equation}
be obtained by concatenating on the left by our fixed path $e_1e_2 \ldots e_{r_w}$. The maps $f_w$ are clearly injective, and since the $(n-1)$st to last vertex of each path in the image of $f_w$ must be $w$, there images are nonintersecting. 

Composing the pasting $\cup{f_w}$ with injections $P_\mathcal{G}(v,n+r_w) \rightarrow P_\mathcal{G}(v,n+k)$
 (which we can do since $\mathcal{G}$ has no stranded states) gives an injection 
 $\cup \mathcal{P}_\mathcal{G}(w,n) \rightarrow P_\mathcal{G}(v,n+k)$, which by 
 (\ref{Eqent1}) and (\ref{Eqent2}) gives 
\begin{equation}
 N_n(\overline{\mathcal{P}}) \leq N^I_{n+k}(\mathcal{P})
\end{equation}
for all $n$. This completes the proof of Theorem ~\ref{entthm}. 
\end{proof}

Theorem \ref{Percor} will follow for general path sets using  a standard formula for the entropy of a sofic shift.
The Perron-Frobenius Theorem says, among other things, that if 
$A$ is a nonnegative, irreducible square matrix then there is a unique real eigenvalue 
$\lambda > 0$ such that $\lambda \geq | \mu |$ for any other eigenvalue $\mu$ of $A$. 
This $\lambda$ is called the \emph{Perron eigenvalue} of $A$ (Lind \cite{Lin84}). 
%************************
% Proposition Perprop
%************************

\begin{prop} \label{Perprop}
 Let $X_\mathcal{G}$ be a one-sided sofic system which is right-resolving, 
 and let $A$ be the adjacency matrix of 
 the underlying (unlabelled) directed graph $G$ of $\mathcal{G}$. Then the topological entropy 
\begin{equation}
 h(X_\mathcal{G}) = \log \lambda, 
\end{equation}
where $\lambda$ is the spectral radius of $A$. If $\mathcal{G}$, and hence $A$, is irreducible,
 then $\lambda$ is the Perron eigenvalue of $A$.
\end{prop}

\begin{proof} This is proved for one-sided shifts of finite type in 
Kitchens (\cite{Kit98}, Observation 1.4.2, p. 23).
For right-resolving one-sided sofic shifts the proof follows by
the same argument. The right-resolving property guarantees that
all paths can be told apart by their symbolic dynamics starting from the
initial state $v_0$., i.e. the path symbols uniquely determine the sequence of vertex
labels along the path.
\end{proof}

%PROOF OF THEOREM 1.9
\begin{proof}[Proof of Theorem \ref{Percor}.] 
The assumption that the path set $\sP$ has a presentation $(\sG, v)$
which is reachable means that the topological entropy of $(\sG, v)$ as a path set
equals the topological entropy of $\sG$ as a one sided sofic-shift.
Now combining Theorem ~\ref{entthm} and Proposition ~\ref{Perprop} yields
 the result. 
  \end{proof}

%**************************************************************************************************
%
% Section 10:  Appendix%
%**************************************************************************************************

\section*{Appendix A.: Minimality of  Path Set Presentations in Example \ref{ex33} }\label{sec10}

The construction in Example \ref{ex33} produced a right-resolving presentation $(\sH, v_0)$ as a path set of
the one-sided sofic shift $Y$ given by the right-resolving graph $\sG:=\sG_n$.
Recall that $|V(\sG)|=n$ while   $|V(\sH)| = 2^n -1$.
Here we show that $(\sH, v_0)$ is a minimal presentation as a path set,
even though the underlying graph is not strongly connected.
That is, given  any other right-resolving presentation  $(\sH', v_1)$  of $Y$ as  a path set, 
our object is to show that $|V(\sH')| \ge 2^n -1$. We present an ad hoc argument specific to
this example.

The label set of $\sG_n$ is $\sL =\{ 0, 1, ..., n-1, n, ..., 2n\}$.
We consider those paths emanating from the pointed vertex $v_0$ 
that have successive labels $\sigma:= ( n+ i_1, ..., n+i_{n-1})$
in which $(i_1, ..., i_{n-1})$ is a subset of $n-1$ distinct values of $(0, 1, ..., n-1)$; there are $n!$ such
paths, and they are all realized in the one-sided sofic shift $Y$. 
We term {\em vertices reached at level $k$} the collection of all vertices reached after exactly  $k$ steps along any
of these paths, for $1 \le k \le n-1$.  We observe that 
such paths never extend to the value $n+i_n$, where $i_n$ is the unique omitted value in
$(0,1,..., n-1)$, but they do extend one step further, allowing $i_n$ to equal any one of $i_1, i_2, ..., i_{n-1}$.
It follows that these  paths must reach at least $n$ distinct vertices at step $n-1$, which are distinguished from each other
by the fact that they each omit an exit edge having value $i_n$ and do not omit an exit edge with any
other value $n+i_j$, $1 \le i \le n-1$. 
Thus we have shown that least $n= \binom{n}{n-1}$ distinct vertices are reached at the final level $n-1$.
We also note that all paths of length $n$ starting from $v$ having labels $(n+i_1, ..., n+i_n)$ with each $1 \le i_j \le n$
and with at least two repeated values are legal paths  occurring in the graph and are
unique paths by the right-resolving property.

We now study  the vertices reached at step $j \ge 0 $ for all these paths, for $0 \le j \le n-1$, and show
by induction on $j$  that at the $j$-th step 
there are at least $\binom{n}{j}$ such vertices  not reached at earlier steps.
This result implies 
$$
|V(\sH')|  \ge \,\sum_{j=0}^{n-1} \binom{n}{j} = 2^n -1,
$$
which proves minimality.
The base case $j=0$ of the induction holds, since  at  step $0$ there is the initial vertex $v_0$.
For the induction step, assume it is true at stage $j-1$,
and consider all the vertices reached at stage $j$. Note that for $j \le n-2$ all these vertices have exit edges for all
$n$ possible labels $n, ..., 2n-1$, a property which distinguishes these vertices from final level vertices.
These vertices also cannot coincide with those from any of the  
earlier levels, because if  in those cases there is a shorter path reaching these vertices $v_1$, which then
extends to a path at the final level vertex using $n-2$ or fewer letters. But such a path is
automatically extendible one more step with an arbitrary choice of label $n+i_n$, $0 \le i \le n-1$,
which contradicts the path having reached a vertex at which extension by one of these letters is forbidden.
This extendability property holds because the right-resolving property requires all such other paths be unique from
the start vertex $v_0$.
It follows that there must be at least $\binom{n}{j}$ distinct such vertices $w$, corresponding to each possible  of $j$ distinct
(unordered)  labels
$i_1, ..., i_j$ 
of $(0, 1, ..., n-1)$. For if two such vertices coincided, we could extend a path to a final level vertex, and
at least one of the two extensions would contain a repeated letter, showing that the final level vertex
has an exit edge labelled $n+i$ for $0 \le i \le n-1$, a contradiction. This completes the 
induction step, and the proof.

%**************************************************************************************************
%
% Section 11:  Appendix B
%**************************************************************************************************

\section*{Appendix B: Characterizations of one-sided sofic shifts}\label{sec11}

This appendix establishes
 the equivalence of our definition \ref{de13} of one-sided sofic shift, to that
used in Ashley, Kitchens and Stafford \cite{AKS92}.  \\

%\begin{thm}
\noindent {\bf Theorem B-1.}
{\em The following are equivalent, for a 
one-sided shift invariant set $S \subset \sA^{\NN}$
\begin{enumerate}
\item
$S$ is a one-sided sofic shift, i.e. it is the set of labels
of all one-sided infinite walks along a finite labelled directed
graph $\sG$, starting from any vertex $\sG$.
\item
$S$ is closed one-sided subshift that has finitely many different infinite follower sets.
\item
$S$ is a closed one-sided subshift that has finitely many different finite follower sets.
\end{enumerate}
}
%\end{thm}

\begin{proof}
Here condition (3) is the definition used in \cite{AKS92}.
For a closed one-sided subshift, given a finite word $W$, we  let $F(W)$ denote the finite
follower set determined by $W$, consisting of all finite words that directly follow an occurrence of $W$.
We let $\overline{F}(W)$ denote the infinite follower set, consisting of all one-sided infinite words
that follow an occurrence of $W$.

$(1) \Rightarrow (3)$. 
By hypothesis the one-sided sofic shift $S$ has a  labelled graph presentation $\sG$,
and this may be regarded as the presentation of a two-sided sofic shift
as well.
Finite follower sets with prefix $W$ are the set of all allowable finite paths 
that follow a path with the given (finite) prefix,  and they are the same
for the one-sided shift as for  the two-sided shift.
The result follows since
the follower set  property (3) is established for two-sided shifts
 (\cite[Theorem 3.2.10]{LM95}).

$(3) \Rightarrow (2)$ 
By hypothesis $S$ is closed and shift-invariant.
Then each finite follower set $\sF(W)$ 
associated to a finite word $W$  determines a
unique infinite follower set $\overline{F}(W)$, which is the set of one-sided infinite words that
are limits of words in the finite follower set. Indeed any infinite string in $\overline{F}(W)$
has all its finite prefixes in $F(W)$ and is therefore in the limit set;  conversely
any limit word of $F(W)$ will be in the infinite follower set,  since $S$ is closed and shift-invariant.
%%%%%%%%%%%%%%%%%%%%%%%%
%Since $S$ is a closed one-sided subshift,
%it contains all these limit words. Conversely, any infinite word in $S$
%must have each of its intital  finite subwords in some follower set, and
%since there are finitely many finite follower sets, one of them contains
%infinitely many initial segments of the word, hence gives rise to the
%infinite word. Thus $\bar{F}(W)$ is the infinite follower set associated to the word $W$.
%%%%%%%%%%%%%%%%%%%%%%%%%
Because the shift $S$ is one-sided, each infinite follower set in it is determined 
by some finite prefix $W$. But there are only finitely different such sets $F(W)$ by (2).

$(2) \Rightarrow (1).$
We construct a directed labelled graph $\sG$ whose vertices correspond to
the distinct infinite follower sets. Given such a set $\overline{F}(W)$, we can add
one letter $a \in \sA$ to it on the right $W' :=W a$ and then draw an edge in the graph $\sG$,
labelled $a$ to the vertex corresponding to the follower set $\overline{F}(W')$.
One now checks that the graph $\sG$ generates the one-sided shift $S$. 
\end{proof}

%\noindent{Remark.} 
Ashley, Kitchens and Stafford  \cite[Lemma 2.8]{AKS92} characterize one-sided sofic shifts $S$
as those closed shift-invariant sets that avoid an infinite set of forbidden blocks describable as
the complete set of finite paths of  a certain finite automaton between a given
marked  initial node $I$ and a given marked  terminal node $T$. 

%************************************************************************
%
%  Bibliography 
%
%
%************************************************************************


\begin{thebibliography}{9}
\bibitem{AL11p}
W. Abram and J. C. Lagarias,
\emph{$p$-adic path set fractals and arithmetic,}
J. Frac. Anal. Geom {\bf 1} (2014), to appear. 
arXiv:1210.2478.
\bibitem{AL11b}
W. Abram and J. C. Lagarias,
\emph{Intersections of multiplicative translates of $3$-adic Cantor sets},
 arXiv:1308:3133.
 
\bibitem{AM79}
R.L. Adler and B. Marcus, 
\textit{Topological entropy and equivalence of dynamical systems},
Memoirs of the American Mathematical Society,  No. 219, 
AMS: Providence, RI 1979.

\bibitem{AL11}
S. Akiyama and B. Loridant,
\emph{Boundary parametrization of self-affine tiles,}
J. Math. Soc. Japan {\bf 63} (2011), no. 2, 525--579.
\bibitem{AS92}
J.P. Allouche and J. O. Shallit,
\emph{The ring of $k$-regular sequences,}
Theor. Comp. Sci. {\bf 96} (1992), 163--197.



\bibitem{AS03}
J.P. Allouche and J. O. Shallit,
\textit{Automatic Sequences: Theory, Applications, Generalizations},
Cambridge University Press: Cambridge 2003.

\bibitem{AKS92}
J. Ashley, B. Kitchens and M. Stafford,
\emph{Boundaries of Markov Partititons,}
Trans. Amer. Math. Soc. {\bf 333} (1992), No. 1, 177--201.

\bibitem{Ban06}
C. Bandt, N. V. Hung and H. Rao, 
\emph{On the open set condition for self-similar fractals,}
Proc. Amer. Math. Soc. {\bf 134} (2006), 1369--1374.

\bibitem{BBLT06}
B. Barat, V. Berth\'{e}, P. Liardet and J. Thuswaldner,
\emph{Dynamical directions in numeration,}
Ann. Inst. Fourier (Grenoble) {\bf 56} (2006), 1987--2092.




\bibitem{Ba93}
M. Barnsley,
{\em Fracals Everywhere.} Second edition. Revised with the assistance of and with  a foreword
by Hawley Rising III.
Academic Press: Boston MA. 1993. 

\bibitem{BP97}
M.-P. B\'{e}al and D. Perrin,
\emph{Symbolic Dynamics and Finite Automata,}
in: \textit{Handbook of Formal Languages, Vol. 2},
G. Rozenberg, A. Salomaa, Eds., Springer-Verlag 1997, pp. 463--506.


\bibitem{Cob69}
A. Cobham,
\emph{On the base-dependence of sets of numbers recognizable by finite automata,}
Math. Systems Theory {\bf 3}  (1969), 186--192.

\bibitem{Cob72}
A. Cobham,
\emph{Uniform Tag Sequences}
Math. Systems Theory {\bf 6}  (1972), 164--192.

\bibitem{Con71}
J. H. Conway,
\textit{Regular Algebra and Finite Machines,}
Chapman and Hall: London 1971.

\bibitem{Edg08}
G. Edgar,
\emph{Measure, topology and fractal geometry, Second Edition}
Springer-Verlag: New York 2008.


\bibitem{Eil74A}
S. Eilenberg,
\textit{Automata, languages and machines. Vol. A},
Academic Press: New York 1974.

\bibitem{FW09}
De-jung Feng and Yang Wang,
{\em On the structures of generating iterated function systems of Cantor sets,}
Adv. Math. {\bf 222} (2009), 1964--1981.

\bibitem{Fis75}
R. Fischer,
\emph{Sofic systems and graphs,}
Monatshefte f\"{u}r Mathematik {\bf 80} (1975), 179--186.

\bibitem{Fis75b}
R. Fischer,
\emph{Graphs and symbolic dynamics,}
Colloq. Math. Soc. J\'{a}nos B\'{o}lyai : Topics in
Information Theory {\bf 16} (1975), 229--243.

\bibitem{Hut81}
J. E. Hutchinson,
\emph{Fractals and self similarity,}
Indiana Univ. Math. Journal {\bf 30} (1981), 713--747.

\bibitem{JM94}
N. Jonoska and B. Marcus,
\emph{Minimal presentations for irreducible sofic shifts,}
IEEE Trans. Information Theory {\bf 40} (1994), No. 6, 1818--1825.
%\bibitem{KH95}
%A. Katok and B. Hasselblatt, 
%\textit{Introduction to the Modern Theory of Dynamical Systems}
% Cambridge University Press, New York, 1995.
 
 \bibitem{KPS11}
R. Kenyon, Y. Peres and B. Solomyak,
\emph{Hausdroff dimension for fractals invariant under the 
multiplicative integers,}
Ergod. Th. Dyn. Sys. , to appear. {\tt arXiv:1102.5136}

\bibitem{Kit98}
B.P. Kitchens.
\textit{Symbolic Dynamics: One-Sided, Two-sided and Countable State Markov Shifts}
Springer-Verlag, Berlin Heidelberg,1998.

\bibitem{Kr83}
W. Krieger,
\emph{On sofic systems I,}
Israel J. Math. {\bf 48} (1984), No. 4, 305--330.

\bibitem{Kr87}
W. Krieger,
\emph{On sofic systems II,}
Israel J. Math. {\bf 60} (1987), No. 2, 167--176.

\bibitem{Lag09}
J.C. Lagarias,
 \emph{Ternary expansions of powers of 2}, 
J. London Math. Soc.(2) {\bf 79} (2009),  562-588.

\bibitem{LN07}
Ka-Sing Lau and Sze-Man Ngai,
\emph{A generalized finite type condition for iterated function systems,}
Advances in Math. {\bf 208} (2007), no. 2, 647--671.

\bibitem{Lin84}
D. Lind,
\emph{The entropies of topological Markov shifts and a related class of algebraic integers},
Ergod. Th. Dyn. Sys. {\bf 4} (1984), no. 2, 283--300.

\bibitem{LM95}
D. Lind and B. Marcus,
 \textit{An Introduction to Symbolic Dynamics and Coding,} 
 Cambridge University Press, New York, 1995. (Reprinted 1999
 with corrections.)

\bibitem{Mah61}
K. Mahler,
\textit{Lectures on diophantine approximations, Part I. $g$-adic numbers and Roth's theorem},
Prepared from notes of R. P. Bambah, University of Notre Dame Press, Notre Dame IN 1961.

\bibitem{MU03}
R. D. Mauldin and M. Urba\'{n}ski,
\emph{Graph directed Markov systems. Geometry and dynamics of limit sets,} 
Cambridge Tracts in Mathematics No. 148, Cambridge Univ. Press: Cambridge 2003.

\bibitem{MW88}
R.D. Mauldin and S.C. Williams, 
\emph{Hausdorff Dimension of Graph Directed Constructions}, 
Transactions of the American Mathematical Society, {\bf 309}, No. 2 (1988) , 811-829.



%\bibitem{LO02}
%M. Lothaire,
%\textit{Algebraic Combinatorics on Words,}
%Cambridge University Press: Cambridge, UK  2002. 

\bibitem{McN66}
R. McNaughton,
\emph{Testing and generating infinite sequences by a finite automaton,}
Information and Control {\bf 9} (1966), 521--530.

\bibitem{NW01}
Sze-Man Ngai and Yang Wang,
\emph{Hausdorff dimension of self-similar sets with overlaps,}
J. London Math. Soc. {\bf 63} (2001), no. 3, 655--672.

\bibitem{PSSS12}
Y. Peres, J. Schmeling, S. Seuret, B. Solomyak,
\emph{Dimensions of some fractals defined via semigroup generated 
by $2$ and $3$,}
eprint: {\tt arXiv:1206.4742}.

\bibitem{PP04}
D. Perrin and J.-\'{E}. Pin,
\emph{Infinite Words. Automata, Semigroups, Logic and Games,}
Elsevier: Amsterdam 2004.



%\bibitem{Sak09}
%J. Sakarovitch,
%\meph{Elements of Automata Theory,}
%Translated from the 2003 French original by Reuben Thomas,
%Cambridge Univ. Press; Cambridge 2009.

%\bibitem{Sim06}
%J.G. Simonsen, 
%\emph{On the Computabillity of the Topological Entropy of Subshifts},
% Discrete Mathematics and Theoretical Computer Science, {\bf 8} (2006), 83-96.

%\bibitem{Th04}
%K. Thomsen,
%\emph{On the structure of a sofic shift space,}
%Trans. Amer. Math. Soc. {\bf 356} (2004), No. 9, 3557--3619.

%\bibitem{Tr01}
%P. Trow,
%\emph{Determining presentations of sofic shifts,}
%Theor. Comp. Sci. {\bf 259} (2001), 199--216.

\bibitem{Wei73}
B. Weiss,
\emph{Subshifts of finite type and sofic systems,}
Monatshefte f\"{u}r Math. {\bf 77} (1973), No. 5, 462--474.

\bibitem{Wil89}
J. C. Willems,
\emph{Models for dynamics,}
in:  \textit{Dynamics Reported, Vol. 2},  pp. 171--269,
Dynam. Report. Ser. Dynam. Systems Appl. 2, 
Wiley: Chichester 1989.

%\bibitem{Wil86}
%S. Williams,
%\emph{A sofic system with infinitely many minimal covers,}
%Proc. Amer. Math. Soc.  {\bf 98} (1986), No. 3, 
\end{thebibliography}
\end{document}